\documentclass[11pt,centertags,final]{amsart}
\usepackage{amsmath,amsthm,amscd,amssymb}
\usepackage{braket,mleftright}
\usepackage{enumerate}
\usepackage{mathabx}
\usepackage{graphicx,caption,subcaption}
\usepackage{cite}
\usepackage[colorlinks=true,linkcolor=blue,urlcolor=blue,citecolor=blue,anchorcolor=blue]{hyperref}

\newcommand{\R}{\mathbb{R}}
\newcommand{\C}{\mathbb{C}}
\newcommand{\N}{\mathbb{N}}
\newcommand{\Z}{\mathbb{Z}}
\newcommand{\E}{\mathcal{E}}
\renewcommand{\H}{\mathcal{H}}

\newcommand{\eg}{\textit{e.g. }}
\newcommand{\cf}{\textit{cf. }}
\newcommand{\ie}{\textit{i.e. }}

\newcommand{\mrm}[1]{\mathrm{#1}}
\newcommand{\co}{\colon}
\newcommand{\abs}[1]{\lvert#1\rvert}
\newcommand{\aaabs}[1]{\left\lvert #1\right\rvert}
\newcommand{\norm}[1]{\lVert#1\rVert}

\DeclareMathOperator{\ot}{\otimes}
\DeclareMathOperator{\op}{\oplus}

\newcommand{\Dom}{\mathop{\mrm{dom}}}
\newcommand{\Ker}{\mathop{\mrm{ker}}}
\newcommand{\re}{\mathop{\mrm{Re}}}
\newcommand{\img}{\mathop{\mrm{Im}}}
\newcommand{\spn}{\mathop{\mrm{span}}}
\newcommand{\Arg}{\mathop{\mrm{Arg}}}
\newcommand{\Res}{\mathop{\mrm{Res}}}
\newcommand{\resolv}{\mathop{\mrm{res}}}
\newcommand{\sing}{\mathop{\mrm{spec_{s}}}}

\newcommand{\HS}{\H_{\mrm{s}}}
\newcommand{\HP}{\H_{\mrm{p}}}
\newcommand{\HSC}{\H_{\mrm{sc}}}
\newcommand{\HC}{\H_{\mrm{c}}}
\newcommand{\HA}{\H_{\mrm{ac}}}

\theoremstyle{plain}
\newtheorem{thm}{Theorem}[section]
\newtheorem{lem}[thm]{Lemma}
\newtheorem{prop}[thm]{Proposition}

\theoremstyle{definition}
\newtheorem{rem}[thm]{Remark}

\allowdisplaybreaks
\numberwithin{equation}{section}
\begin{document}
\title[Spectrum of a family of spin-orbit coupled Hamiltonians]{Spectrum of a family 
of spin-orbit coupled Hamiltonians with singular perturbation}
\author{R.~Jur\v{s}\.{e}nas}
\address{Vilnius University, Institute of Theoretical Physics and Astronomy,  
A. Go\v{s}tauto 12, Vilnius 01108, Lithuania}
\email{Rytis.Jursenas@tfai.vu.lt}
\date{\today}
\begin{abstract}
The present study is the first such attempt to examine rigorously and comprehensively the spectral properties 
of a three-dimensional ultracold atom when both the spin-orbit interaction and the Zeeman field are taken into 
account. The model operator is the Rashba spin-orbit coupled operator in dimension three. The self-adjoint extensions
are constructed using the theory of singular perturbations, where regularized rank two perturbations 
describe spin-dependent contact interactions. The spectrum of self-adjoint extensions
is investigated in detail laying emphasis on the effects due to spin-orbit coupling. When the spin-orbit-coupling
strength is small enough, the asymptotics of eigenvalues is obtained. The conditions for the existence of
eigenvalues above the threshold are discussed in particular. 
\end{abstract}
\maketitle
\section{Introduction}
The spin-orbit coupled quantum systems are described by the Hamiltonian
which is realized as the differential operator in the tensor product space
$\R^{3}\ot\C^{2}$:
\begin{equation}
S:=-\Delta \ot I
+i\alpha (\nabla_{1}\ot\sigma_{2}-\nabla_{2}\ot\sigma_{1})
+\beta I\ot \sigma_{3} \quad
(\alpha,\beta\geq0). 
\label{eq:DiffOp}
\end{equation}
In (\ref{eq:DiffOp}), $\alpha$ stands for the spin-orbit-coupling strength,
$\beta$ is the strength of the magnetic Zeeman field \cite{Dalibard11}.
The symbol $\Delta$ denotes the three-dimensional Laplace operator,
$\nabla_{j}$ ($j=1,2$) is the gradient in the $j$th component of a
three-dimensional position-vector. The standard Pauli matrices are 
represented by $(\sigma_{j})_{j=1,2,3}$. The imaginary unit
$i\equiv\sqrt{-1}$. Here and elsewhere we write $I$ for an identity operator. 
The space in which $I$ acts is expected to be understood from the context.
\subsection{Main goal}
The main purpose of the current exposition is to study the spectrum
of self-adjoint extensions of the symmetric operator constructed from
(\ref{eq:DiffOp}). Let $S$ be defined on a set of compactly 
supported smooth functions, with the support outside the origin $0\in\R^{3}$:
\begin{equation}
\Dom S=C_{0}^{\infty}(\R^{3}\backslash\{0\})\ot\C^{2}.
\label{eq:DiffOp1}
\end{equation}
Then $S$ is symmetric with respect to the scalar product in the Hilbert tensor product $L^{2}(\R^{3})\ot\C^{2}$. 
In applications, the self-adjoint extensions of so defined $S$ are referred to as the Rashba spin-orbit coupled 
Hamiltonians considered in the presence of the \textit{out-of-plane} magnetic field, with the impurity scattering 
treated via the spin-dependent contact interaction. A good review on various types of disorder in condensed-matter 
systems is given in \cite{Modugno10}. For the analysis of long-range interactions in electronic systems, as opposed 
to the present discussion where the zero-range interactions are studied, the reader may refer to 
\cite{Chesi11,Ambrosetti09}.
\subsection{Special cases}
The special case $\alpha=0$ was discussed recently in \cite{Cacciapuoti09,Cacciapuoti07}, where,
among other things, the authors examined the effects due to the \textit{in-plane} magnetic field.
The corresponding term in (\ref{eq:DiffOp}) is then replaced by a unitarily equivalent one $\beta I\ot\sigma_{1}$.
In their Theorem~2 in \cite{Cacciapuoti09} the authors showed under what conditions the point spectrum is empty. 
Here we prove that under exactly the same conditions there is a subfamily of self-adjoint extensions whose point
spectrum is $\{\pm\beta\}$.

The case $\alpha=\beta=0$ is a classic one. The book \cite{Albeverio05} is much
used as a reference book for a standard (or von Neumann) operator extension theory; see also an extensive list of
references therein. The extensions constructed using spaces of boundary values are discussed, for example, in
\cite{Albeverio13,Malamud12,Mikhailets99,Derkach91}. A modern and comprehensive review of various operator 
techniques is given in \cite{Bruning08}. 

In one and two spatial dimensions, the spectral properties of self-adjoint extensions of $S$, for $\alpha,\beta\geq0$, 
were studied in \cite{Jursenas13,Carlone11,Cacciapuoti09,Cacciapuoti07,Bruning07,Exner07}. A mean field 
interpretation, which is commonly accepted in physics literature, can be found in \cite{Kim15,Hu13,Liao13,Liu13}.
\subsection{Motivation and main tools}
A recently proposed technique \cite{Anderson13} (see also \cite{Wang12,Cheuk12,Campbell11}) for producing the 
Rashba-type spin-orbit coupling for a three-dimensional ultracold atom serves as our main motive to examine a general 
case $\alpha,\beta \geq0$. Many more motivating aspects concerning the dynamics of quantum particles can be found in 
\cite{Cacciapuoti09,Cacciapuoti07,Cacciapuoti10}, and we do not repeat them here. Contrary to \cite{Cacciapuoti09}, 
where the modified Krein resolvent formula \cite{Albeverio05-1} has been used as a starting point for calculating the 
resolvent of $S$ (with $\alpha=0$), here, we exploit the theory of singular finite rank perturbations 
\cite{Simon05,Dijksma05,Kurasov03,Albeverio02,Albeverio00,Albeverio97}. On the one hand, the two approaches lead 
to an identical description of the boundary conditions that ensure self-adjointness \cite{Albeverio07}. On the other 
hand, the singularities of Green function (see \eg \cite{Albeverio10} and references therein) are successfully eliminated 
using the so-called renormalization procedure \cite{Kurasov03}. Such a procedure is implemented in the theory of 
singular perturbations in its most natural way \cite{Albeverio00}. 

As an illustration, let us consider a formal "operator" $S+C\delta$, where $S$ in (\ref{eq:DiffOp}) is written in matrix 
representation, and where $\delta$ is the Dirac distribution concentrated at the origin; for $\alpha=\beta=0$, this type 
of operator, when restricted to one spatial dimension but extended to the many-body case, is studied in 
\cite{Albeverio01}. The matrix $C$, which is Hermitian and invertible, plays the role of the coupling parameter of 
contact interaction. The Lipmann--Schwinger equation for the perturbed operator shows that the solution of the 
integral equation is inconsistent, in that the Green function for $S$ is singular at $x=0$ (see (\ref{eq:frg}) for the 
details), meaning that the equation does not have a unique solution. The phenomenon is usually called the ultraviolet 
divergence. For, the Fourier integral of the Green function in the momentum representation diverges when $x=0$. To 
avoid divergences, typically one introduces an auxiliary UV momentum cut-off and regularizes the interaction 
parameter. The details can be found in \cite[Sec.~3.7]{Braaten08}, see also 
\cite{Takei12,Vyasan11a,Vyasan11,Ozawa11}.
In the theory of singular perturbations, however, one deals with a regularization and a renormalization 
procedure applied to $\delta$. In this case the regularization is meant in the sense that the Dirac distribution is 
extended to some subspace of a Hilbert space. Roughly speaking the action of $\delta$ on the function which is 
undefined at the origin is replaced by the action of some regularized $\delta$. The present discussion is justified 
rigorously in Sec.~\ref{sec:self}. Next, the regularized and renormalized $\delta$ is determined by the choice of an 
Hermitian matrix $R$, which is usually called \textit{admissible}; the matrix realization of $R$ is due to an extra 
$\C^{2}$ in (\ref{eq:DiffOp}). It turns out that the theory also deals with an additional parameter (\ie the matrix $R$). 
However, the precise operator realization of $S+C\delta$ by extending $\delta$ gives one an advantage over the cut-off 
practice in an obvious way. One shows that the self-adjoint realizations of $S+C\delta$ are parametrized in terms of 
$\Gamma:=-C^{-1}-R$, with the exception of the trivial extension corresponding to $C=0$ and the Friedrichs extension 
corresponding to $C^{-1}=0$. For a general $\alpha,\beta\geq0$, the functional $\delta$ is of the class for which $R$ is 
not unique, but in some cases $R$ can be found exactly; \eg $R=-(4\sqrt{2}\pi)^{-1}\mathbb{I}$ for $\alpha=\beta=0$.  
For some special classes of singular finite rank perturbations, the uniqueness of $R$ is discussed in great
detail in \cite{Hassi09}.

Another reason for choosing the approach of singular perturbations is that the scattering theory for singular finite rank 
perturbations is fully established \cite[Chap.~4]{Albeverio00}: Using our results, the scattering matrix can 
be found directly from \cite[Eq.~(4.34)]{Albeverio00}. For this particular reason we do not 
examine scattering states but rather concentrate on analytic properties of the spectrum.
\subsection{Solvability of the model}
To the best of our knowledge, the model studied in the current paper is investigated by means of the theory of singular 
perturbations for the first time. The lack of rigorous results concerning the spectral analysis of the perturbed $S$ may 
be explained by the complexity of the Green function for the free Hamiltonian (or else the trivial extension of $S$). 
A typical strategy is to write the Green function in the momentum representation. Then the dispersion relation and the 
essential spectrum are easy to deduce. However, the computation of the inverse Fourier transform of Green function is a 
rather troublesome task. Recently \cite{Jursenas14}, a hypergeometric series representation for the Green function in 
the coordinate representation has been derived. In contrast, the analogous result in two spatial dimensions possesses a 
closed form \cite{Bruning07,Carlone11,Exner07}.
\small
\begin{figure}
\centering
\begin{subfigure}[b]{0.35\textwidth}
\includegraphics[width=\textwidth]{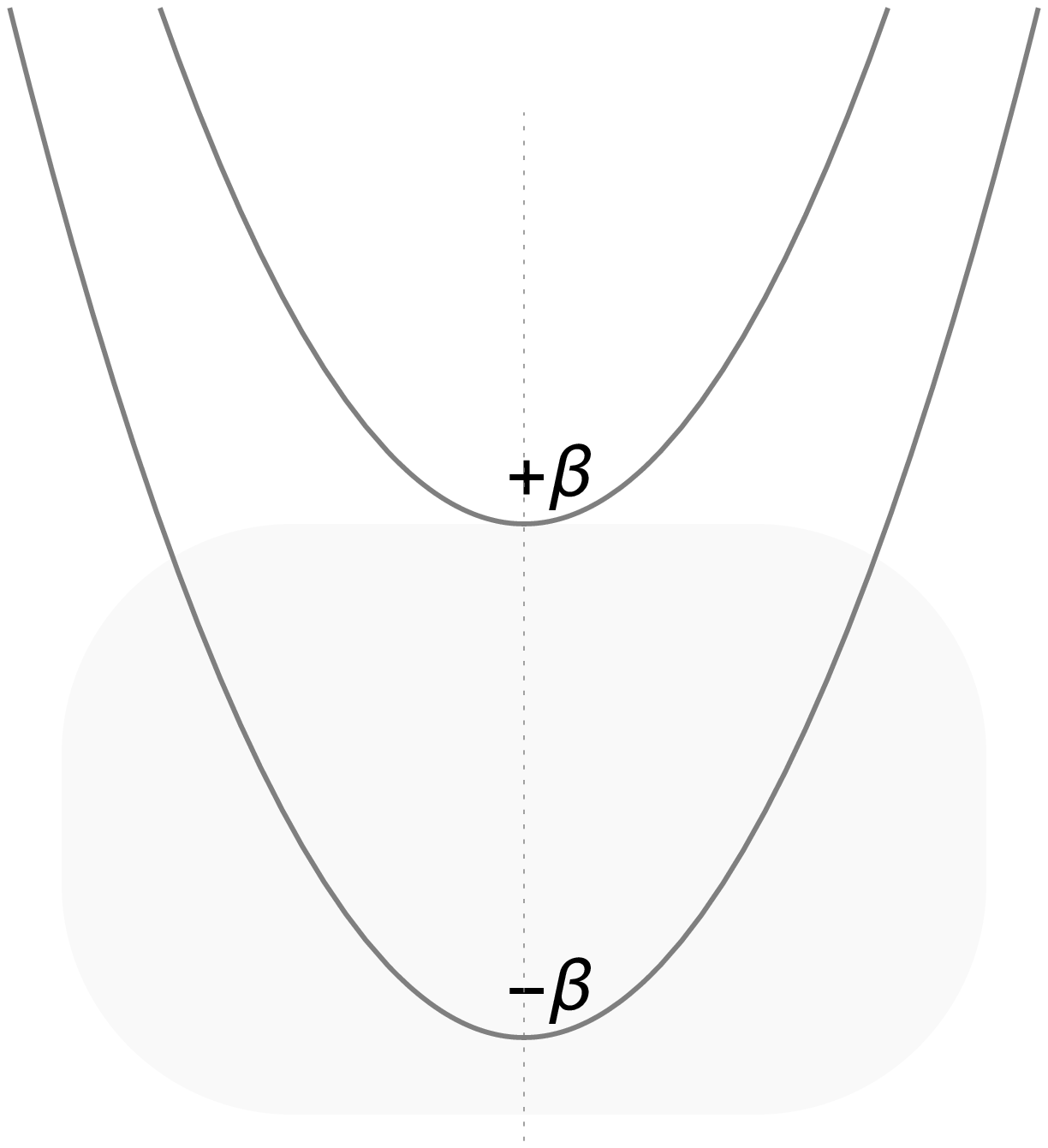}
\caption{$\alpha=0$, $\Sigma=\beta\geq0$}
\label{fig:A}
\end{subfigure}\hfill
\begin{subfigure}[b]{0.35\textwidth}
\includegraphics[width=\textwidth]{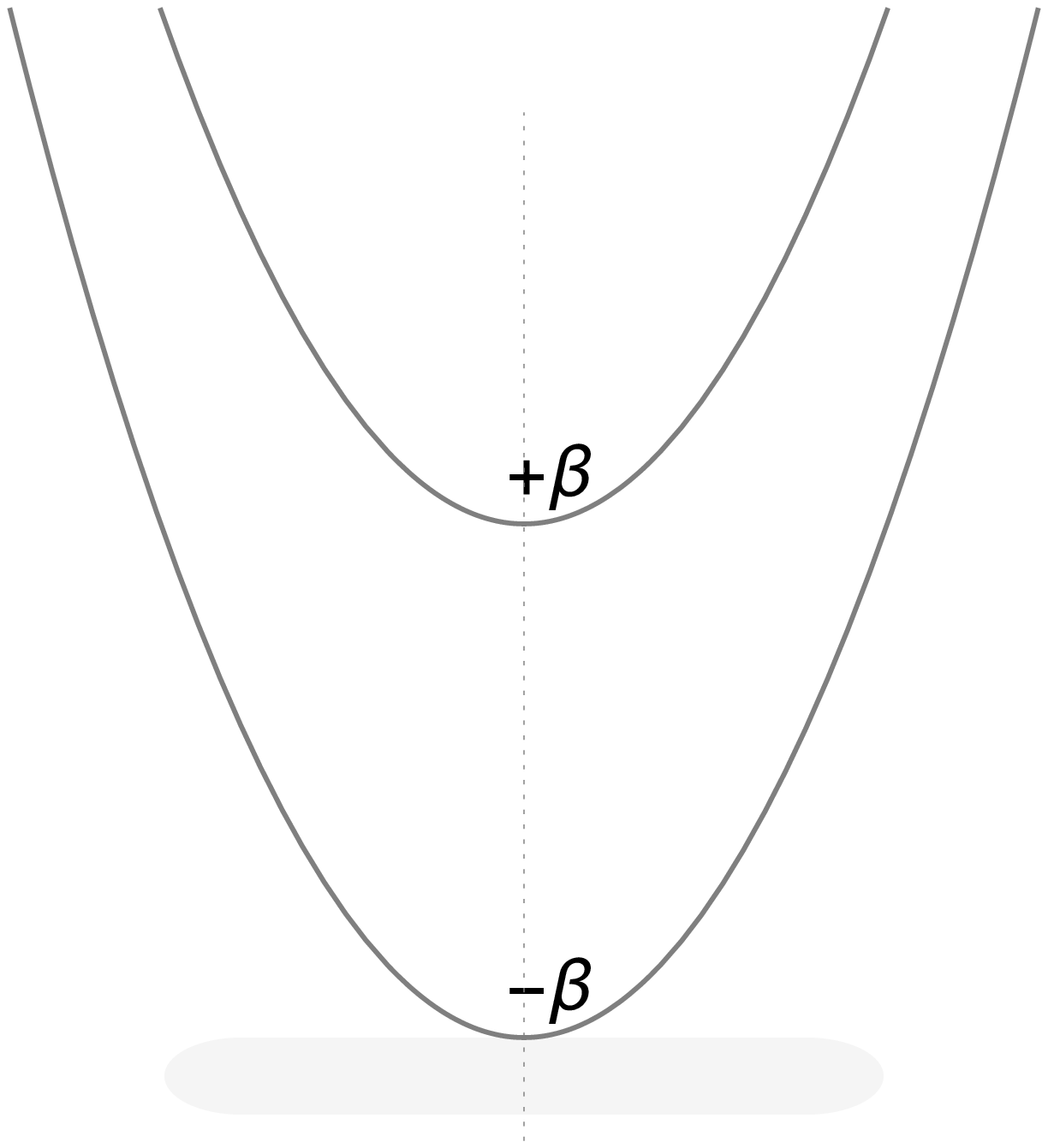}
\caption{$0<\alpha<\sqrt{2\beta}$, $\Sigma=\beta>0$}
\label{fig:B}
\end{subfigure}\hfill
\begin{subfigure}[b]{0.35\textwidth}
\includegraphics[width=\textwidth]{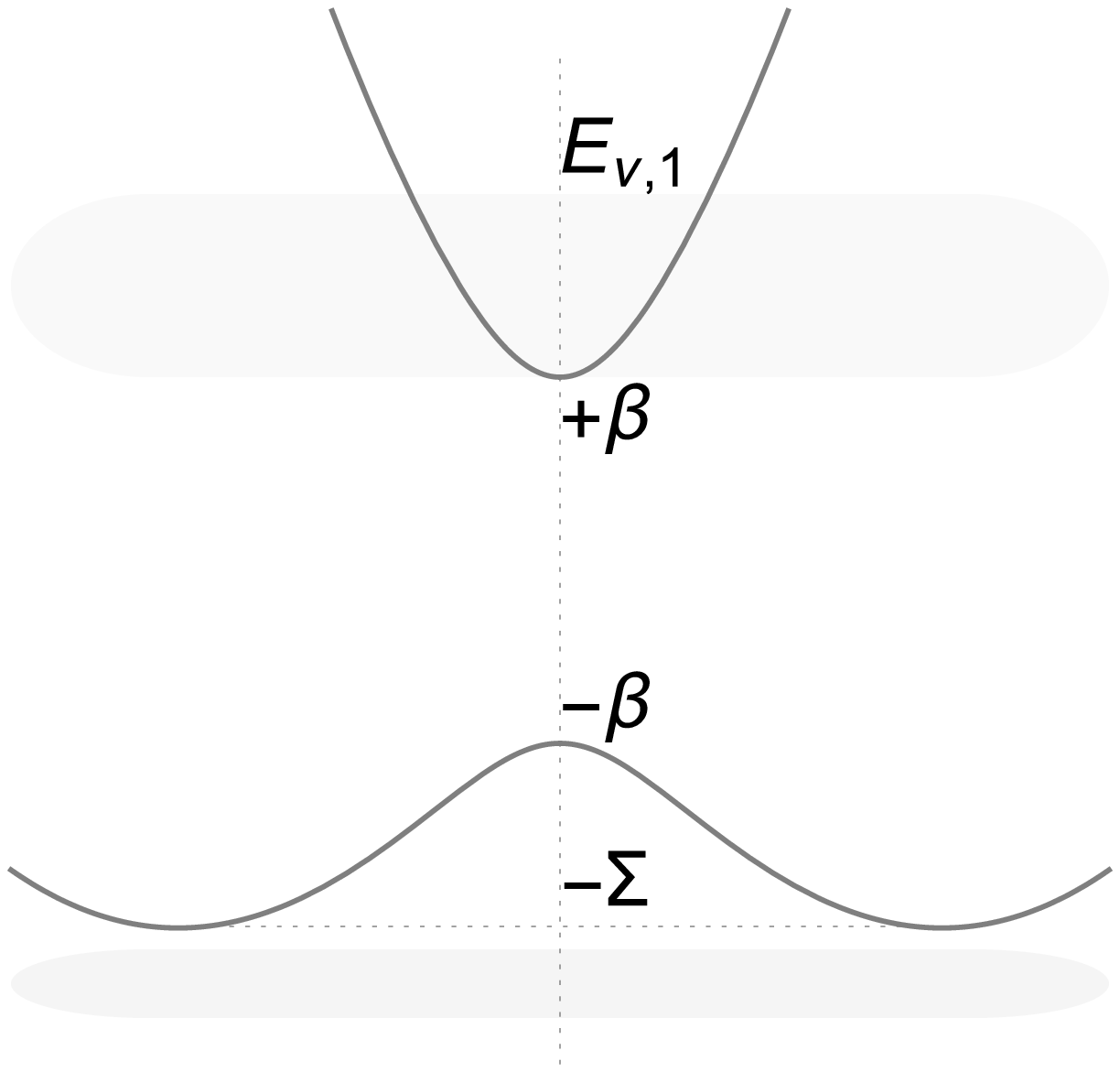}
\caption{$0<\sqrt{2\beta}\leq\alpha$, $\beta\leq\Sigma\leq1$}
\label{fig:C}
\end{subfigure}
\caption{Schematic of projected dispersion relations. Permitted eigenvalues are situated along
the vertical axis within the gray area. For example, (b) shows that the point spectrum is empty
above $-\beta$. The details are listed in Theorems~\ref{thm:without}, \ref{thm:existence}, \ref{thm:existence3}.}
\label{fig:ABC}
\end{figure}
\normalsize

The absolute convergence of hypergeometric series causes the limitation on the parameters $\alpha,\beta\geq0$. 
The result is that the infimum of the continuous spectrum $[-\Sigma,\infty)$ must satisfy the condition 
$-\Sigma\geq-1$. 
Here $\Sigma$ equals $\beta$ if $\beta>\alpha^{2}/2$ and $(\beta/\alpha)^{2}+(\alpha/2)^{2}$ otherwise. 

For $\alpha$ arbitrarily small or zero, the restriction can be removed by analytic continuation.
Using the resolvent formula for a perturbed Hamiltonian (or else nontrivial self-adjoint extension of $S$),
we find that the eigenvalues solve the transcendental equation. We then derive the solutions analytically 
with the accuracy up to $O(\alpha^{4})$. In this particular case we observe a tendency that the spin-orbit
term moves the eigenvalues down from the threshold $-\beta$. In general, the main conclusions following from
our results that concern eigenvalues can be schematically described as shown in Fig.~\ref{fig:ABC}: 
For $\alpha>0$ small, there are no eigenvalues above the threshold $-\beta$ no matter the form of a nonzero coupling 
parameter of contact interaction ($C$). For $\alpha$ large, there exists a subfamily of Hamiltonians having 
eigenvalues in $[\beta,E_{\nu,1}]$, for some lucid $E_{\nu,1}\geq\beta$, but none of Hamiltonians has
eigenvalues in $[-\Sigma,\beta)$.
\section{Self-adjoint extensions}\label{sec:self}
Let $H$ be a lower semi-bounded self-adjoint operator in a separable Hilbert space $\H_{0}$.
Let $(\H_{n})_{n\in\Z}$ be the scale of Hilbert spaces associated to $H$. The scalar product and the induced norm
in $\H_{n}$ are denoted by $\braket{\cdot,\cdot}_{n}$ and $\norm{\cdot}_{n}$, respectively.
Consider the orthonormal system $(\psi_{s})_{s=\pm}\in \H_{-2}\backslash\H_{-1}$ and define
the duality pairing $\braket{\cdot,\cdot}\co \H_{-n}\times\H_{n}\to\C$, for $n\in\N_{0}$, in a usual way 
\cite[Sec.~1.2.2]{Albeverio00}. The orthogonality, which is assumed with respect to the scalar product 
$\braket{\cdot,\cdot}_{-2}^{*}:=\braket{(H-i)^{-1}\cdot,(H-i)^{-1}\cdot}_{0}$ in $\H_{-2}$, implies the one with
respect to $\braket{\cdot,\cdot}_{-2}$. We write $\psi$ for the column matrix $(\psi_{s})$; $\braket{\psi,\cdot}$ is the
column matrix $(\braket{\psi_{s},\cdot})$. 

Let $H^{0}$ be a symmetric and densely defined restriction of $H$ to the domain
\begin{equation}
\Dom H^{0}=\{f\in \H_{2}\co \braket{\psi,f}=0\}.
\label{eq:H0}
\end{equation}
For $z\in\resolv H$ (the resolvent set of $H$), we define the functions
\begin{equation}
\Phi_{s}(z):=(H-z)^{-1}\psi_{s}
\label{eq:DefEl0}
\end{equation}
and the column matrices $\Phi\equiv(\Phi_{s})$. The functions
$(\Phi_{s}(i))_{s=\pm}$ are orthonormal with respect to the scalar product $\braket{\cdot,\cdot}_{0}$ in $\H_{0}$,
for it holds $\braket{\Phi_{s}(i),\Phi_{s^{\prime}}(i)}_{0}=\braket{\psi_{s},\psi_{s^{\prime}}}_{-2}^{*}$ ($s,s^{\prime}
=\pm$). Let $H^{0\,*}$ be the adjoint of $H^{0}$. Then $\Phi_{s}(z)\in\Ker(H^{0\,*}-z)$ and $H^{0\,*}$ 
has deficiency indices (2,2).
Considering $\Dom H^{0\,*}$ as a Banach space $\H_{2}\dotplus\C^{2}$ with the norm
\[\norm{f}_{\Dom H^{0\,*}}^{2}=\norm{\tilde{f}}_{2}^{2}+\sum_{s}\abs{b_{s}(f)}^{2}\]
($f\in \Dom H^{0\,*}$; $\tilde{f}\in \H_{2}$; $b_{s}(f)\in \C$), we write
\begin{equation}
f=\tilde{f}+\tfrac{1}{2}\sum_{s}b_{s}(f)(\Phi_{s}(i)+\Phi_{s}(-i)).
\label{eq:fDomH*}
\end{equation}
The column matrix $b(f)\equiv(b_{s}(f))$, and we regard $b_{s}(\cdot)$ as a functional
$\Dom H^{0\,*}\to\C$.

Let $\psi_{s}^{R}$ be an extension of $\psi_{s}$ to $\Dom H^{0\,*}$ \cite{Hassi09,Albeverio97}:
\begin{equation}
\braket{\psi^{R},f}=\braket{\psi,\tilde{f}}+Rb(f),\quad
\braket{\psi^{R},\tilde{f}}=\braket{\psi,\tilde{f}}
\label{eq:psiext}
\end{equation}
where $f$ and $\tilde{f}$ are as above, and where $\psi^{R}\equiv(\psi_{s}^{R})$. The matrix 
$R\equiv(R_{ss^{\prime}})\in\C^{2\times2}$ is the admissible matrix for the functionals of 
class $\H_{-2}\backslash\H_{-1}$. Then $\psi_{s}^{R}$ is an element of the topological dual
$(\Dom H^{0\,*})^{\prime}=\H_{-2}\dotplus\C^{2}$ equipped with the supremum norm
$\sup\abs{\braket{\psi_{s}^{R},f}}$; the supremum is taken over $s=\pm$ and 
$f\in\Dom H^{0\,*}$ such that $\norm{f}_{\Dom H^{0\,*}}\leq1$.

Using another decomposition $\Dom H^{0\,*}=\Dom H^{0}\dotplus \spn\{\Phi_{s}(\pm i)\}$, we have that,
for $f\in\Dom H^{0\,*}$,
\begin{equation}
f=f^{0}+\sum_{s}(a_{s}^{+}(f)\Phi_{s}(i)+a_{s}^{-}(f)\Phi_{s}(-i))
\label{eq:fDomH*0}
\end{equation}
($f^{0}\in \Dom H^{0}$; $a_{s}^{\pm}(f)\in\C$). But then, using (\ref{eq:fDomH*}) and (\ref{eq:fDomH*0}),
\begin{equation}
\tilde{f}=f^{0}+\tfrac{1}{2}\sum_{s}ib_{s}^{\prime}(f)\left(\Phi_{s}(i)-\Phi_{s}(-i)\right)
\label{eq:tf}
\end{equation}
($b_{s}^{\prime}(\cdot)$ is regarded as a functional $\Dom H^{0\,*}\to\C$;
the column matrix $b^{\prime}\equiv(b_{s}^{\prime})$) and
\begin{equation}
\begin{pmatrix}
b_{s}(f) \\ b_{s}^{\prime}(f)
\end{pmatrix}=
\begin{pmatrix}
1 & 1 \\ -i & i
\end{pmatrix}
\begin{pmatrix}
a_{s}^{+}(f) \\ a_{s}^{-}(f)
\end{pmatrix}.
\label{eq:bbp}
\end{equation}
Using the normalization
of $(\Phi_{s}(i))_{s=\pm}$, equations~(\ref{eq:fDomH*0}) and (\ref{eq:bbp}), and
\begin{equation}
H^{0\,*}f=H\tilde{f}+\tfrac{1}{2}\sum_{s}ib_{s}(f)
\left(\Phi_{s}(i)-\Phi_{s}(-i)\right)
\label{eq:Hadjf}
\end{equation}
one shows that $(\C^{2},b,b^{\prime})$ is a boundary triple for $H^{0\,*}$.

The self-adjoint extensions $H(\Gamma)$ of $H^{0}$ in $\H_{0}$ are in one-to-one correspondence with
an Hermitian matrix $\Gamma\equiv(\Gamma_{ss^{\prime}})\in \C^{2\times 2}$,
and they are defined as restrictions of $H^{0\,*}$ to the domain \cite{Albeverio00}
\begin{equation}
\Dom H(\Gamma)=\{f\in\Dom H^{0\,*}\co \braket{\psi^{R},f}=-C^{-1}b(f)\}
\label{dom:AK}
\end{equation}
($\Gamma:=-C^{-1}-R$; $C\in \C^{2\times2}$; $\det C\neq0$) provided that $C^{-1}\neq0$. 
When the coupling parameter $C\equiv0$, we set $H(\Gamma)\equiv H$. In view of (\ref{eq:fDomH*}), 
(\ref{eq:Hadjf}), (\ref{dom:AK}), $H(\Gamma)$ is regarded as 
\begin{equation}
H(\Gamma)=H+V, \quad V\co f\mapsto \sum_{s}(C\braket{\psi^{R},f})_{s}\psi_{s}
\quad(f\in \Dom H(\Gamma))
\label{eq:HV}
\end{equation}
so that $V$ is a singular rank two perturbation. 
The resolvent operator of $H(\Gamma)$ is given by
\begin{equation}
(H(\Gamma)-z)^{-1}=(H-z)^{-1}+\sum_{s}((\Gamma-Q(z))^{-1}
\braket{\Phi(\bar{z}),\cdot}_{0})_{s}\Phi_{s}(z)
\label{eq:resolvAK}
\end{equation}
($z\in \resolv H(\Gamma)=\resolv H\cap\{z\in\C\co \det(\Gamma-Q(z))\neq0\}$).
Here $Q\equiv(Q_{ss^{\prime}})\in\C^{2\times2}$ is the Krein $Q$-matrix function
\begin{equation}
Q_{ss^{\prime}}(z):=\Braket{\psi_{s},\frac{I+zH}{H-z}\frac{1}{H^{2}+I}\psi_{s^{\prime}}}
\quad(s,s^{\prime}=\pm).
\label{eq:w3}
\end{equation}

Let us show that (\ref{dom:AK}) does not apply to the case $C^{-1}=0$. Let formally $C^{-1}=0$. Then 
$\braket{\psi^{R},f}=0$ for all $f\in\Dom H(\Gamma=-R)$. But then, since there 
always exists $f\in\Dom H^{0\,*}$ such that $\braket{\psi^{R},f}\neq0$, it follows from the Hahn--Banach theorem that 
$\Dom H(-R)$ is not dense in $\Dom H^{0\,*}$, and hence in $\H_{0}$. This contradicts the fact that $H^{0}$
is densely defined and thus its all self-adjoint extensions are operators.

Let $H_{F}:=H(-R)\vert\Ker b$. Using (\ref{eq:H0}), (\ref{eq:tf}), and the orthogonality of $(\Phi_{s}(i))_{s=\pm}$,
$\braket{\psi,\tilde{f}}=-b^{\prime}(f)$. Using (\ref{eq:psiext}), 
$H_{F}=H$, $\Dom H_{F}=\{f\in \Dom H^{0\,*}\co b(f)=b^{\prime}(f)=0\}$.
\begin{prop}\label{prop:Infty2}
$H_{F}$ is the Friedrichs extension of $H^{0}$.
\end{prop}
\begin{proof}
Let $H_{F}^{\prime}$ be the Friedrichs extension of $H^{0}$. We show that $H_{F}=H_{F}^{\prime}$.

Let $\mathfrak{s}$, with $\Dom \mathfrak{s}=\Dom H^{0}$, be the form associated with $H^{0}$:
\[\mathfrak{s}[f,g]:=\braket{H^{0}f,g}_{0}\quad(f,g\in \Dom H^{0}).\]
Then its closure $\bar{\mathfrak{s}}$ is defined on $\Dom\abs{H^{0}}^{1/2}$, \ie 
$\Dom \bar{\mathfrak{s}}=\{f\in \H_{1}\co \braket{\psi,f}=0\}$,
and, by definition, $H_{F}^{\prime}:=H^{0\,*}\vert \Dom \bar{\mathfrak{s}}$. 
Put (\ref{eq:fDomH*}) in $\braket{\psi,f}=0$, apply $\psi_{s}\notin\H_{-1}$, and deduce that
$b(f)=0$ and $\braket{\psi,\tilde{f}}=0$. But $\braket{\psi,\tilde{f}}=-b^{\prime}(f)$ and hence
$H_{F}^{\prime}=H^{0\,*}\vert\Ker b\cap \Ker b^{\prime}$ as required.
\end{proof}
\begin{rem}
According the theory of boundary triplets (see \eg \cite[Chap.~14]{Schmudgen12}), the self-adjoint
extensions of $H^{0}$ are in one-to-one correspondence with the self-adjoint relations on $\C^{2}$.
One can show that, for $C^{-1}\neq0$ and $C\neq0$, the extensions correspond to the relations whose range is of 
dimension $1$ and $2$. When $C=0$, the range of the corresponding relation on $\C^{2}$ is $\{0\}$, \ie
$H=H^{0\,*}\vert \Ker b$. When $C^{-1}=0$, the relation is $\{(0,0)\}$, \ie $H_{F}=H^{0\,*}\vert
\Ker b\cap \Ker b^{\prime}$.
\end{rem}
\section{Free Hamiltonian}
Throughout, $H^{m}$ ($m=1,2$) is the Sobolev space of 
$L^{2}$-functions whose distributional derivatives of order 
$\leq m$ belong to $L^{2}$. The closure of $C_{0}^{\infty}$ in
$H^{m}$-norm is $H_{0}^{m}$.

The closure of the operator $S$ defined on the domain (\ref{eq:DiffOp1}) is
a densely defined symmetric operator
\begin{equation}
H^{0}:=S\lvert H_{0}^{2}(\R^{3}\backslash\{0\})\ot\C^{2}.
\label{eq:FreeH}
\end{equation}
Equivalently,
\begin{equation}
H^{0}=S,\quad \Dom H^{0}=\{f\in\Dom H\co f(0)=0\} 
\label{eq:FreeH1} 
\end{equation}
where
\begin{equation}
H:=S\lvert H^{2}(\R^{3})\ot\C^{2}.
\label{eq:FreeH2}
\end{equation}
The operator $H$ in (\ref{eq:FreeH2}) is the maximal operator associated with $S$, and it is the closure of the 
operator $S\lvert C_{0}^{\infty}(\R^{3})\ot\C^{2}$.
By Gauss and Green formulas (see \eg \cite[Appendix~D]{Schmudgen12}) one deduces that 
the formal adjoint of $S\lvert C_{0}^{\infty}(\R^{3})\ot\C^{2}$
coincides with itself. Hence the closure of $S\lvert C_{0}^{\infty}(\R^{3})\ot\C^{2}$
is the minimal operator associated with $S$. Now that the adjoint of the minimal (resp. maximal) operator 
coincides with the maximal (resp. minimal) operator, it follows that $H$ is self-adjoint. We call $H$ the \textit{free
Hamiltonian}. By using exactly the same arguments one concludes that
the adjoint operator $H^{0\,*}$ of $H^{0}$, (\ref{eq:FreeH})--(\ref{eq:FreeH1}), is defined as follows:
\begin{equation}
H^{0\,*}=S\lvert H^{2}(\R^{3}\backslash\{0\})\ot\C^{2}.
\label{eq:AdjH}
\end{equation}

The free Hamiltonian $H$ has lower bound $-\Sigma$, where $\Sigma$ equals $\beta$ if 
$\beta>\alpha^{2}/2$ and $(\beta/\alpha)^{2}+(\alpha/2)^{2}$ otherwise. The Green function
for $H$, or the \textit{free Green function}, possesses a representation
\begin{align}
(H-z)^{-1}(x)=&
G_{2}(x;z)\ot I-\beta G_{1}(x;z)\ot\sigma_{3} 
+i\alpha\left(\nabla_{2}G_{1}(x;z)\ot\sigma_{1} \right. \nonumber \\
&-\left.\nabla_{1}G_{1}(x;z)\ot\sigma_{2}\right)\quad
(x\in\R^{3}\backslash\{0\})
\label{eq:frg}
\end{align}
provided that $z\in \resolv H$ meets at least one of the following:
\begin{enumerate}[\upshape (a)]
\item\label{item:a} $2\beta>\alpha^{2}$ and
$\displaystyle\beta\leq\abs{z}<2(\beta/\alpha)^{2}$
\item\label{item:b} $\abs{z}\geq\Sigma$; the equality is available
only if $0\leq2\beta<\alpha^{2}$
\item\label{item:c} $\displaystyle\abs{z}>\max\{\beta/(2\sqrt{R}),\alpha^{2}/(4S)\}$ and 
$\displaystyle R+\left(S-\tfrac{1}{2}\right)^{2}=\tfrac{1}{4}$
($R,S>0$).
\end{enumerate}
We always assume that $z\in \resolv H$ satisfies at least one of (\ref{item:a})--(\ref{item:c}) unless explicitly
stated otherwise. In particular, $(H-i)^{-1}(\cdot)$ possesses a series representation (\ref{eq:frg}) if $\Sigma\leq1$ 
(put $\abs{z}=1$ in (\ref{item:a})--(\ref{item:c})).
The functions $G_{j}(x;z)\equiv G_{j}(x;\alpha,\beta,z)$ ($j=1,2$) obey a hypergeometric series representation,
and they are studied in \cite{Jursenas14}. Let 
\[G_{2}^{\mrm{ren}}(x;z)\equiv G_{2}^{\mrm{ren}}(x;\alpha,\beta,z):=G_{2}(x;z)-
\frac{e^{-\abs{x}\sqrt{-z}}}{4\pi\abs{x}}.\]
When $x=0$, one assumes the limit $\abs{x}\downarrow0$ on the right. We have that 
\begin{subequations}\label{eq:G1X10renX10}
\begin{align}
G_{1}(0;\alpha,\beta, z)=&
\frac{1}{4\pi\alpha}\mrm{artanh}
\left(\frac{\alpha}{\beta}\sqrt{ \frac{- z}{2}
\left(1-\sqrt{ 1-\left(
\frac{\beta}{ z} \right)^{2} } \right) }\right), 
\label{eq:G1X0} \\
G_{2}^{\mrm{ren}}(0;\alpha,\beta, z)=&
\frac{\sqrt{- z}}{4\pi}-\frac{1}{4\pi}
\sqrt{ \frac{- z}{2}
\left(1+\sqrt{ 1-\left(
\frac{\beta}{ z} \right)^{2} }\right)} \nonumber \\
&+\frac{ \alpha }{ 8\pi }
\mrm{artanh}
\left(\frac{\alpha}{\beta}\sqrt{ \frac{- z}{2}
\left(1-\sqrt{ 1-\left(
\frac{\beta}{ z} \right)^{2} } \right) }\right).
\label{eq:G2renX0}
\end{align}
\end{subequations}
When $ z\in\resolv H$, the inverse hyperbolic tangent 
$\mrm{artanh}$ is defined entirely on $\C\backslash\{-1,1\}$.
Let $U_{\epsilon}$ be an $\epsilon$-neighborhood
of the origin $0\in\R^{3}$. For $x\in U_{\epsilon}$ and for $\epsilon>0$ small,
\begin{equation}
\nabla_{j}G_{1}(x;\alpha,\beta, z)=
-\frac{\hat{x}_{j}}{8\pi},\quad
\hat{x}_{j}:=\frac{x_{j}}{\abs{x}}\quad
(j=1,2).
\label{eq:nablajG1X0}
\end{equation}
\section{Orthonormal functionals}
Let $\delta$ be the Dirac distribution concentrated at $0\in\R^{3}$. Define the singular distribution
\[\psi_{s}:=\mathcal{N}_{s}\delta\ot g_{s} \quad (s=\pm),\quad \text{where}\quad 
g_{+}:=\begin{pmatrix}
1 \\ 0
\end{pmatrix},\quad 
g_{-}:=\begin{pmatrix}
0 \\ 1
\end{pmatrix}.\]
The normalization constant $\mathcal{N}_{s}>0$ is defined as
\[\mathcal{N}_{s}:=\left(\img G_{s}^{\mrm{ren}}(0;i)+\frac{1}{4\sqrt{2}\pi}\right)^{-1/2}.\]
Here we write, for simplicity, 
\begin{align*}
G_{s}^{\mrm{ren}}(x;z)\equiv G_{s}^{\mrm{ren}}(x;\alpha,\beta,z):=&
G_{s}(x;z)- \frac{e^{-\abs{x}\sqrt{-z}}}{4\pi\abs{x}} \quad (x\in\R^{3};\,z\in \resolv H), \\
G_{s}(x;z)\equiv G_{s}(x;\alpha,\beta,z):=&
G_{2}(x;z)- s\beta G_{1}(x;z) \quad (x\in\R^{3}\backslash\{0\};\,z\in \resolv H).
\end{align*}
For $\alpha=\beta=0$, $\mathcal{N}_{s}=2\sqrt[4]{2}\sqrt{\pi}$ is as in
\cite[Sec.~2.3]{Albeverio00}, where the authors examine the Laplace operator in 
dimension three (recall that $S=-\Delta\ot I$ for $\alpha=\beta=0$) with the interaction determined using
the model of generalized perturbations.

In general, the relation $\mathcal{N}_{s}>0$ can be shown as follows. Observe that
\begin{subequations}\label{eq:stp5}
\begin{align}
\img G_{1}(0;i)=&
\frac{ 1 }{ 8\pi\alpha }\left(\Arg\left(1+c(1+i)\right)-\Arg\left(1-c(1+i)\right)\right), \\
\img G_{2}^{\mrm{ren}}(0;i)=&
-\frac{1}{4\sqrt{2}\pi}+\frac{1}{8\pi}\sqrt{ 1+\sqrt{ 1+\beta^{2} } } \nonumber \\
&+\frac{\alpha}{16\pi}
\left(\Arg\left(1+c(1+i)\right)-\Arg\left(1-c(1+i)\right)\right).
\end{align}
\end{subequations}
Here we define, for convenience,
\[c:=\frac{\alpha}{2\beta}\sqrt[4]{ 2+\beta^{2}-2\sqrt{ 1+\beta^{2} } } \quad
(c\geq0),\]
and $\Arg$ is the principal value of the argument; the range of $\Arg$ is in
$(-\pi,\pi]$. Then
\begin{equation}
\mathcal{N}_{ s}^{-2}=
\frac{ 1 }{ 8\pi }
\left(\sqrt{1+\sqrt{1+\beta^{2}}}+
\left(\frac{\alpha}{2}-\frac{ s\beta}{\alpha}\right) 
\left(\Arg\left(1+c(1+i)\right)-\Arg\left(1-c(1+i)\right)\right) \right).
\label{eq:phinorm*}
\end{equation}
Seeing that
\begin{align*}
&\Arg\left(1+c(1+i)\right)-\Arg\left(1-c(1+i)\right)
\geq0\,(=0\,\text{iff}\,\alpha=0), \\
&\frac{\alpha}{2}-\frac{ s\beta}{\alpha}\geq0\quad
\text{iff}\quad s=-1\quad\text{or}\quad s=1\quad\text{and}\quad
\beta\leq\frac{\alpha^{2}}{2}, \\
&\frac{\alpha}{2}-\frac{ s\beta}{\alpha}<0\quad
\text{iff}\quad s=1\quad\text{and}\quad
\beta>\frac{\alpha^{2}}{2},
\end{align*}
we conclude from (\ref{eq:phinorm*}) that the condition
$\mathcal{N}_{ s}^{-2}>0$ is equivalent to the condition
\begin{align*}
\sqrt{1+\sqrt{1+\beta^{2}}}>&f(\alpha,\beta)
\quad \text{for}\quad \beta>\frac{\alpha^{2}}{2}, \quad\text{where} \\
f(\alpha,\beta):=&\left(\frac{\beta}{\alpha}-\frac{\alpha}{2}\right)
\left(\Arg\left(1+c(1+i)\right)-\Arg\left(1-c(1+i)\right)\right).
\end{align*}
But, for a fixed $\beta\geq0$,
\[f(\alpha,\beta)\leq f(0,\beta)=
\sqrt[4]{ 2+\beta^{2}-2\sqrt{ 1+\beta^{2} } }
<\sqrt{1+\sqrt{1+\beta^{2}}}.\]
Therefore, $\mathcal{N}_{ s}^{-2}>0$ and this last estimate implies that we can always choose
$\mathcal{N}_{s}>0$.

Next, using (\ref{eq:DefEl0}), 
\begin{equation}
\Phi_{s}(z)\equiv\Phi_{s}(x;z)=\mathcal{N}_{s}\left(G_{s}(x;z)\ot g_{s}+ 
s\alpha D_{s}G_{1}(x;z)\ot g_{-s}\right)\quad(x\in \R^{3}\backslash\{0\})
\label{eq:fsigma}
\end{equation}
with arguments and parameters as in (\ref{eq:DefEl0}). Here $D_{s}:=\nabla_{1}+is\nabla_{2}$.
Below we prove two equivalent relations: 
\begin{itemize}
\item $(\Phi_{s}(i))_{s=\pm}$ is an orthonormal system in $\H_{0}=L^{2}(\R^{3})\ot\C^{2}$
\item $(\psi_{s})_{s=\pm}$ is an orthonormal system in $\H_{-2}\backslash\H_{-1}$
($\H_{-n}=H^{-n}(\R^{3})\ot\C^{2}$; $n=1,2$).
\end{itemize}
The important conclusion is that, if the above relations hold true, then the operator $H^{0}$ in 
(\ref{eq:FreeH})--(\ref{eq:FreeH1}) may be treated similar to the operator $H^{0}$ in (\ref{eq:H0}); 
subsequently, the functions $(\Phi_{s}(\pm i))_{s=\pm}$ in (\ref{eq:fsigma}) form deficiency subspaces 
of the adjoint operator $H^{0\,*}$ defined in (\ref{eq:Hadjf}), (\ref{eq:AdjH}).

Using (\ref{eq:DefEl0}) and 
\[\frac{ H }{ H^{2}+I }=\tfrac{1}{2}\left(\frac{ 1 }{ H-i }+\frac{ 1 }{ H+i }\right),\]
we have that 
\[\Braket{\psi_{s},\frac{ H }{ H^{2}+I}\psi_{s} }
=\tfrac{1}{2}\braket{\psi_{s},\Phi_{s}(i)+\Phi_{s}(-i)}.\]
But
\begin{align*}
\abs{\braket{\psi_{s},\Phi_{s}(i)+\Phi_{s}(-i)}}=&
2\mathcal{N}_{s}^{2}\abs{\re G_{s}(0;i)} \\
=&2\mathcal{N}_{s}^{2}
\aaabs{ \re G_{s}^{\mrm{ren}}(0;i)+\lim_{\abs{x}\downarrow0}\re 
\frac{ e^{ -\abs{x}\sqrt{-i} } }{ 4\pi\abs{x} } }
=\infty.
\end{align*}
Then
\[\aaabs{ \Braket{\psi_{s},\frac{ H }{ H^{2}+
I }\psi_{s} }}\leq2\norm{ \psi_{s} }_{-1}^{2}\]
(\cf \cite[Proof of Theorem~3.1]{Albeverio97-1}) implies that $\psi_{s}\notin\H_{-1}$.
Next, we show that 
\begin{equation}
\braket{\Phi_{s}(z),\Phi_{s^{\prime}}(z)}_{0}
=\frac{\delta_{ss^{\prime}}\mathcal{N}_{s}^{2}}{\img z}
\img\left(G_{s}^{\mrm{ren}}(0;z)-\frac{\sqrt{-z}}{4\pi}\right)
\label{eq:Phizorth}
\end{equation}
($s,s^{\prime}=\pm$; $z\in\resolv H$; $\img z\neq0$; $\delta_{ss^{\prime}}$ is the Kronecker symbol). 
When $\img z=0$, take the limit $\img z\to0$ in (\ref{eq:Phizorth}):
\begin{align}
\braket{\Phi_{s}(z),\Phi_{s^{\prime}}(z)}_{0}=&
\frac{ \delta_{ss^{\prime}}\mathcal{N}_{s}^{2} }{ 
8\sqrt{2}\pi\sqrt{ z^{2}-\beta^{2} } } \nonumber \\
&\times 
\left(\sqrt{ -z+\sqrt{ z^{2}-\beta^{2} } }+
\frac{ \beta\left(\alpha^{2}-2s\beta\right)
\sqrt{ -z-\sqrt{ z^{2}-\beta^{2} } } }{ 
2\beta^{2}+\alpha^{2}\left(z+\sqrt{ z^{2}-\beta^{2} }\right) }\right)
\label{eq:Phizorth2}
\end{align}
($s,s^{\prime}=\pm$; $z<-\Sigma$).

Since
\[\frac{1}{H-\bar{z}}\frac{1}{H-z}=\frac{1}{2i\img z}\left(\frac{1}{H-z}-\frac{1}{H-\bar{z}}\right)\]
it follows from (\ref{eq:DefEl0}) that
\[\braket{\Phi_{s}(z),\Phi_{s^{\prime}}(z)}_{0}=
\frac{1}{2i\img z}\braket{\psi_{s},\Phi_{s^{\prime}}(z)-\Phi_{s^{\prime}}(\bar{z})}.\]
Then, using (\ref{eq:fsigma}),
\begin{equation}
\braket{\Phi_{s}(z),\Phi_{s^{\prime}}(z)}_{0}=
\frac{\mathcal{N}_{s}\mathcal{N}_{s^{\prime}}}{\img z} 
\lim_{\abs{x}\downarrow0}\left(\delta_{ss^{\prime}}
\img G_{s}(x;z)-s\alpha\delta_{s,-s^{\prime}}D_{-s}\img G_{1}(x;z)\right).
\label{eq:stp2}
\end{equation}
But
\begin{align}
\lim_{\abs{x}\downarrow0}\img G_{s}(x;z)=&
\frac{1}{2i}\lim_{\abs{x}\downarrow0}\left(G_{s}(x;z)-G_{s}(x;\bar{z})\right) 
\nonumber \\
=&
\img G_{ s}^{\mrm{ren}}(0; z)
+\frac{1}{2i}\lim_{\abs{x}\downarrow0}\left(
\frac{ e^{ -\abs{x}\sqrt{- z} } }{ 4\pi\abs{x} }
-\frac{ e^{ -\abs{x}\sqrt{-\bar{ z}} } }{ 4\pi\abs{x} }\right) \nonumber \\
=&\img G_{ s}^{\mrm{ren}}(0; z)-\frac{\img\sqrt{- z}}{4\pi}
\label{eq:stp3}
\end{align}
and, by (\ref{eq:nablajG1X0}),
\begin{align}
\lim_{\abs{x}\downarrow0}D_{ s}\img G_{1}(x; z)=&
\frac{1}{2i}\lim_{\abs{x}\downarrow0}D_{ s}\left(G_{1}(x; z)-G_{1}(x;\bar{ z})\right) 
\nonumber \\
=&\frac{1}{2i}\lim_{\abs{x}\downarrow0}
\left(-\frac{ \hat{x}_{1}+i s\hat{x}_{2} }{ 8\pi }
+\frac{ \hat{x}_{1}+i s\hat{x}_{2} }{ 8\pi }\right)
=0.
\label{eq:stp4}
\end{align}
Put (\ref{eq:stp3}) and (\ref{eq:stp4}) in (\ref{eq:stp2}) and deduce (\ref{eq:Phizorth}). Since $\H_{-1}\subset\H_{-2}$
densely, we conclude that $(\psi_{s})_{s=\pm}$ is an orthonormal system in $\H_{-2}\backslash\H_{-1}$.
\section{Spectrum}\label{sec:sp}
Here we apply the resolvent formula (\ref{eq:resolvAK}) to the spectral analysis of the operator $H(\Gamma)$
regarded as a perturbed Hamiltonian $H+V$ and where the perturbation $V$ is referred to as the spin-dependent 
contact interaction (recall (\ref{eq:HV})). We assume that $\Gamma\neq-R$.
When $\Gamma=-R$, the corresponding operator is $H_{F}$ (Proposition~\ref{prop:Infty2}). 
Since the spectrum of $H_{F}$ coincides
with that of $H$, and the spectrum of $H$ is $[-\Sigma,\infty)$ and it is absolutely continuous, we mainly concentrate
on the spectrum of $H(\Gamma)$. We adopt the classification of the spectrum in 
\cite[Chap.~9]{Schmudgen12}, \cite[Sec.~VII.3]{Reed80}.

The \textit{continuous spectrum} of $H(\Gamma)$ is given by $[-\Sigma,\infty)$. This follows from an invariance of the 
continuous component of the spectrum under singular finite rank perturbations \cite[Theorem~4.1.4]{Albeverio00}
and from the resolvent formula.

The real and isolated singularities of the resolvent operator of $H(\Gamma)$
coincide with the points $E\in\R$ that solve $\det(\Gamma-Q(E))=0$.
Therefore, the \textit{singular spectrum} of $H(\Gamma)$ consists of the points $E\in\R$ such that
\begin{equation}
\det(\Gamma-Q(E))=0.
\label{eq:singularsp}
\end{equation}
The Krein $Q$-matrix function, which is defined in (\ref{eq:w3}), is a diagonal matrix 
$Q(z)\equiv(Q_{ss^{\prime}}(z))$ whose entries are given by
\begin{equation}
Q_{ s s^{\prime}}(z)=\delta_{ s s^{\prime}}
\mathcal{N}_{s}^{2}\left(\frac{1}{4\sqrt{2}\pi}-\frac{\sqrt{-z}}{4\pi}
+G_{s}^{\mrm{ren}}(0; z)-\re G_{s}^{\mrm{ren}}(0;i)\right)
\label{eq:KreinQ}
\end{equation}
($s, s^{\prime}=\pm$). To prove (\ref{eq:KreinQ}), notice that
\[\frac{ I+ z H }{ H- z } \frac{1}{ H^{2}+I }=
\frac{1}{H- z}-\frac{1}{2}\frac{1}{H-i}-\frac{1}{2}\frac{1}{H+i}.\]
Then, using (\ref{eq:DefEl0}) and (\ref{eq:w3}), 
\[Q_{ s s^{\prime}}( z)=
\Braket{\psi_{ s},\Phi_{ s^{\prime}}( z)-
\tfrac{1}{2}\Phi_{ s^{\prime}}(i)-\tfrac{1}{2}\Phi_{ s^{\prime}}(-i) }.\]
By (\ref{eq:fsigma}), it follows that
\begin{align*}
Q_{ s s^{\prime}}( z)=&
\lim_{\abs{x}\downarrow0}\left[
\delta_{ s s^{\prime}}
\mathcal{N}_{ s}^{2}
\left(G_{ s}(x; z)-\re G_{ s}(x;i)\right)\right. \\
&\left.- s\alpha\delta_{ s,- s^{\prime}}
\mathcal{N}_{ s}\mathcal{N}_{ s^{\prime}}
D_{- s}\left(G_{1}(x; z)-\re G_{1}(x;i)\right)\right].
\end{align*}
But 
\begin{align*}
\lim_{\abs{x}\downarrow0}\left(G_{ s}(x; z)-\re G_{ s}(x;i)\right)
=&G_{ s}^{\mrm{ren}}(0; z)-\re G_{ s}^{\mrm{ren}}(0;i) \\
&+\lim_{\abs{x}\downarrow0}\left(
\frac{ e^{ -\abs{x}\sqrt{- z} } }{ 4\pi\abs{x} }
-\re \frac{ e^{ -\abs{x}\sqrt{-i} } }{ 4\pi\abs{x} }\right) \\
=&G_{ s}^{\mrm{ren}}(0; z)-\re G_{ s}^{\mrm{ren}}(0;i) 
+\frac{1}{4\sqrt{2}\pi}-\frac{\sqrt{- z}}{4\pi}
\end{align*}
and
\[\lim_{\abs{x}\downarrow0}
D_{- s}\left(G_{1}(x; z)-\re G_{1}(x;i)\right)=0\]
which proves (\ref{eq:KreinQ}). In particular, $\img Q_{ss}(z)=\norm{ \Phi_{s}(z) }_{0}^{2}\img z$
and $Q_{ss}(\pm i)=\pm i$. For an arbitrary $z\in\C$, but for $\alpha=\beta=0$,
$Q_{ss}(z)=1-\sqrt{-2z}$ coincides with the $Q$-function obtained in \cite[Sec.~2.3.1]{Albeverio00}.
\begin{lem}\label{lem:eigenf}
The eigenspace $\Ker(H(\Gamma)-E)$ consists of eigenfunctions 
\begin{equation}
f(E+i\epsilon)=\sum_{s}b_{s}(f)\Phi_{s}(E+i\epsilon)
\label{eq:eigenf}
\end{equation}
where the coefficients $(b_{s}(f))_{s=\pm}$ are found from the
boundary condition defining the domain of $H(\Gamma)$. For $E<-\Sigma$, $\epsilon=0$. 
For $E\geq-\Sigma$, $\epsilon>0$ and
$f(E+i\epsilon)$ is understood in the generalized sense, \ie $f(E+i\epsilon)$
approaches $f(E)$ strongly in $\H_{0}$ as $\epsilon\downarrow0$.
\end{lem}
\begin{proof}
The proof of the lemma for $E<-\Sigma$ is a natural generalization of results posted
in \cite[Sec.~2.3.1]{Albeverio00}, so we concentrate on the case $E\geq-\Sigma$.

For convenience, define $E_{\epsilon}:=E+i\epsilon$. Let 
$f(E)\in \Ker(H(\Gamma)-E)$. Then $f(E)$ possesses representations (\ref{eq:fDomH*}),
(\ref{eq:fDomH*0}) with an appropriate boundary condition. Plugging (\ref{eq:fDomH*}) and (\ref{eq:Hadjf}) with 
$\tilde{f}\equiv\tilde{f}(E)$ into the eigenvalue equation, we obtain the equation
\begin{equation}
(H-E)\tilde{f}(E)=
\tfrac{1}{2}\sum_{ s}b_{ s}(f)\left(
(E-i)\Phi_{ s}(i)+(E+i)\Phi_{ s}(-i)\right).
\label{eq:eigen1}
\end{equation}
By hypothesis, the operator $(H-E)^{-1}$ is unbounded in $\H_{0}$
and it cannot be applied to (\ref{eq:eigen1}). We apply $(H-E_{\epsilon})^{-1}$ instead.
Define
\begin{subequations}\label{eq:eigen2a}
\begin{equation}
\tilde{f}_{\epsilon}(E):=(H-E_{\epsilon})^{-1}(H-E)\tilde{f}(E).
\label{eq:eigen2a-1} 
\end{equation}
Put (\ref{eq:eigen1}) in (\ref{eq:eigen2a-1}):
\begin{equation}
\tilde{f}_{\epsilon}(E)=\tfrac{1}{2}\sum_{s}b_{s}(f)(H-E_{\epsilon})^{-1}
\left((E-i)\Phi_{s}(i)+(E+i)\Phi_{s}(-i)\right).
\label{eq:eigen2a-2}
\end{equation}
\end{subequations}
Let $f(E_{\epsilon})$ be as in (\ref{eq:eigenf}). Similar to (\ref{eq:eigen2a-1}), define
\begin{equation}
f_{\epsilon}(E):=(H-E_{\epsilon})^{-1}(H-E)f(E).
\label{eq:eigen4a}
\end{equation} 
Use $f(E)$ in (\ref{eq:fDomH*}) and plug the representation into (\ref{eq:eigen4a}).
Then apply (\ref{eq:DefEl0}) and (\ref{eq:eigen2a}):
\begin{align*}
f_{\epsilon}(E)=&
(H-E_{\epsilon})^{-1}(H-E)\left(
\tilde{f}(E)+\tfrac{1}{2}\sum_{s} b_{s}(f)
\left(\Phi_{s}(i)+\Phi_{s}(-i)\right)\right) \\
=&\tilde{f}_{\epsilon}(E)+
\tfrac{1}{2}\sum_{s} b_{s}(f)
(H-E_{\epsilon})^{-1}(H-E)
\left(\Phi_{s}(i)+\Phi_{s}(-i)\right) \\
=&\tfrac{1}{2}\sum_{s}b_{s}(f)
(H-E_{\epsilon})^{-1}
\left((E-i)\Phi_{ s}(i)+(E+i)\Phi_{ s}(-i)\right) \\
&+\tfrac{1}{2}\sum_{ s} b_{ s}(f)
(H-E_{\epsilon})^{-1}(H-E)
\left(\Phi_{ s}(i)+\Phi_{ s}(-i)\right) \\
=&\tfrac{1}{2}\sum_{ s}b_{ s}(f)
(H-E_{\epsilon})^{-1}(\psi_{ s}+\psi_{ s})
=f(E_{\epsilon}).
\end{align*}
Let $\E$ be the resolution of the identity for $H$, and let $\mu(\cdot):=\braket{f(E),\E(\cdot)f(E)}_{0}$
be the spectral measure. Using (\ref{eq:eigen4a}) and $f_{\epsilon}(E)=f(E_{\epsilon})$,
\[\norm{ f(E_{\epsilon})-f(E) }_{0}=
\epsilon\norm{ (H-E_{\epsilon})^{-1}f(E) }_{0}\]
or equivalently,
\begin{equation}
\norm{ f(E_{\epsilon})-f(E) }_{0}^{2}=
\epsilon\img F(E_{\epsilon}),\quad
F(z):=\int_{\R}
\frac{ \mu(d\nu) }{ \nu-z }\quad (z\in \resolv H).
\label{eq:eigen5}
\end{equation}
By \cite[Theorem~11.6-(iii)]{Simon05}, the Borel transform $F$ of $\mu$ fulfills
\begin{equation}
\mu(\{E\})=\lim_{\epsilon\downarrow0}
\epsilon\img F(E+i\epsilon).
\label{eq:eigen5d}
\end{equation}
Since $\E(\{E\})=0$ for all $E\geq-\Sigma$,
we conclude that $\mu(\{E\})=0$. Put this value in (\ref{eq:eigen5d}),
and then deduce the result stated in the lemma from (\ref{eq:eigen5}).
\end{proof}
\begin{thm}\label{thm:sc}
The singular continuous part is absent from the spectrum of $H(\Gamma)$. 
\end{thm}
\begin{proof}
Consider the singular spectrum as a closed set of points $E_{n}\in\R$ that solve 
(\ref{eq:singularsp}) for $n=1,2,\ldots,N$, for some $N\in\N$. Let $\HS$, $\HP$, and $\HSC$
be the reducing subspaces of $\H_{0}$ corresponding to the singular, discontinuous, and
singular continuous parts of $H(\Gamma)$. Then $\HS=\HP\op\HSC$. For each function $f\in\HS$ ,
one finds functions $g\in\HP$ and $h\in\HSC$ such that $f=g+h$. Then the proof is
accomplished if one shows that $h=0$ for each such $f$.

Using the resolvent formula of $H(\Gamma)$, the integral representation
of the resolution of the identity $\E(\Gamma;\cdot)$ \cite[Lemma~XVIII.2.31]{Dunford88},
and the residue theorem, we find that
\begin{equation}
\E(\Gamma;\{E\})f=-\sum_{s}\sum_{n}\chi(\Gamma;E)
\Res_{z=E_{n}}((\Gamma-Q(z))^{-1}\braket{\Phi(E_{n}),f}_{0})_{s}\Phi_{s}(E_{n}+i0)
\label{eq:E}
\end{equation}
for arbitrary $f\in\H_{0}$. Here $\chi(\Gamma;\cdot)$ is the characteristic function of the singular spectrum.
When deriving (\ref{eq:E}) we used $\E(\{E\})=0$ for the resolution of the identity $\E$ of $H$.
We write $\Phi_{s}(E_{n}+i0)$ to ensure that
the right hand side of (\ref{eq:E}) exists for $E_{n}\geq-\Sigma$.
By Lemma~\ref{lem:eigenf}, one safely replaces $\Phi_{s}(E_{n}+i0)$
with $\Phi_{s}(E_{n})$ in the spectral measure
$\braket{f,\E(\Gamma;\{E\})f}_{0}$. For $E_{n}<-\Sigma$, 
$\Phi_{s}(E_{n}+i0)=\Phi_{s}(E_{n})$.

The characteristic function in (\ref{eq:E}) indicates for which points
$E\in\R$ the function $\E(\Gamma;\{E\})f$ is nonzero. These are
the points that belong to the singular spectrum. Since out task is to examine
the functions in $\HS$, we identify $E$ with an arbitrary $E_{n}$.
Equation~(\ref{eq:E}) is valid for an arbitrary $f\in\H_{0}$, and now we take $f\in\HS$.
For each $E_{n}$, there exists $f\in\HS$ such that $\E(\Gamma;\{E_{n}\})f=f$.
If otherwise, since $f\in\HS$, there must exist a Lebesgue null set $e\subset\R$
such that $e$ is absent in the singular spectrum and $\E(\Gamma;e)f=f$. But $\chi(\Gamma;E)=0$ 
for all $E\in e$, and hence $\E(\Gamma;e)f=0$ by (\ref{eq:E}); this contradicts the definition of $\HS$.
Next, for each $f\in\HS$, there exist $g\in\HP$ and $h\in\HSC$ such that $f=g+h$. By Lemma~\ref{lem:eigenf}, 
$g\in\spn\{\Phi_{s}(E_{n}+i0)\co n=1,\ldots,N\}$. Therefore,
since $\HSC$ is the orthogonal complement of $\HP\subseteq\HS$, one has that
$\braket{\Phi_{s}(E_{n}),h}_{0}=0$ for all $n=1,\ldots,N$.
But then, using (\ref{eq:E}), $\E(\Gamma;\{E_{n}\})h=0$.
Recalling that $\E(\Gamma;\{E_{n}\})f=f$ and $f=g+h$, it follows from the latter
that $\braket{h,\E(\Gamma;\{E_{n}\})f}_{0}=\norm{h}_{0}^{2}=0$.
Hence $f\in\HS\Rightarrow f=g\in\HP$, and the proof of the theorem is accomplished.
\end{proof}
It follows from Theorem~\ref{thm:sc} that the \textit{point spectrum} of $H(\Gamma)$ coincides with the
singular spectrum. Using the integral representation of the resolution of the identity and the resolvent formula,
one shows that, for $e$ in the set of Borel subsets of $\R$, and for $f\in\H_{0}$, the spectral measures 
of $H(\Gamma)$ and $H$ agree up to sets of Lebesgue measure zero. Since the absolutely 
continuous part of a measure is supported by the set of $E\in\R$ such that 
$\img\braket{f,(H(\Gamma)-E-i0_{+})^{-1}f}_{0}\neq0$ (see \eg \cite{Bruning08}), we conclude that
the \textit{absolutely continuous} parts of spectral measures of $H(\Gamma)$ and $H$ coincide. 
If we let $\HC$ and $\HA$
be the reducing subspaces of $\H_{0}$ corresponding to the continuous and absolutely continuous parts
of $H(\Gamma)$, then we arrive at the same conclusion by noting that $\HS=\HP$ and $\HA=\HS^{\bot}$,
hence $\HC=\HP^{\bot}=\HA$.
\section{Analytic properties of singular spectrum}
According to the results of the previous paragraph, the notions of singular spectrum and point spectrum
(eigenvalues) are used interchangeably. As a rule, we prefer the notion of singular spectrum to that of eigenvalues, 
because the singular spectrum consists of the singular points $E\in \R$ of the resolvent operator of $H(\Gamma)$
and because the singular points are solutions of equation (\ref{eq:singularsp}), which is our main target in the present 
section. 

Define
\begin{align}
\omega_{ s}:=&
4\pi(\tilde{\Gamma}_{ s s}+\Lambda_{ s}), \quad 
\gamma:=(4\pi\abs{ \tilde{\Gamma}_{+-} })^{2}, \quad
\tilde{\Gamma}_{ s s^{\prime}}:=
\frac{\Gamma_{ s s^{\prime}}}{\mathcal{N}_{ s}
\mathcal{N}_{ s^{\prime}}}, \nonumber \\
\Lambda_{ s}:=&
\re G_{ s}^{\mrm{ren}}(0;i)-\frac{1}{4\sqrt{2}\pi}\quad
(\omega_{ s},\Lambda_{ s}\in\R;\,s, s^{\prime}=\pm;\,\gamma\geq0).
\label{eq:gammaomega}
\end{align}
The matrix $\tilde{\Gamma}\equiv(\tilde{\Gamma}_{ s s^{\prime}})\equiv
\Bigl(\begin{smallmatrix}\tilde{\Gamma}_{++} & \tilde{\Gamma}_{+-} \\
\tilde{\Gamma}_{-+} & \tilde{\Gamma}_{--} \end{smallmatrix}\Bigr)$ is Hermitian. Let also
\begin{equation}
\xi(E):=\frac{1}{\beta}
\sqrt{ \frac{-E}{2}\left(1-
\sqrt{ 1-\left(\frac{\beta}{E} \right)^{2} } \right) }\quad
(E\in\R;\,\beta\geq0).
\label{eq:xif}
\end{equation}
When $\beta=0$, one takes the limit
$\beta\downarrow0$ on the right hand side: $\xi(E)=1/(2\sqrt{-E})$.
Using (\ref{eq:gammaomega}) and (\ref{eq:xif}), (\ref{eq:singularsp}) obeys the following form:
\begin{equation}
\gamma=\prod_{s}
\left(\omega_{ s}\pm\frac{1}{2\xi(E)}-
\left(\frac{\alpha}{2}-\frac{ s\beta}{\alpha}\right)
\mrm{artanh}(\alpha\xi(E))\right).
\label{eq:sing}
\end{equation}
The upper (resp. lower) sign is taken if $E<\beta$ (resp. $E\geq\beta$).
Owing to $\mrm{artanh}$, equation~(\ref{eq:sing}) is transcendental and does not possess analytic solutions.
For the reason that we are unable to solve equation (\ref{eq:sing}) 
as it stands, we examine (\ref{eq:sing}) in three specific cases:
\begin{enumerate}[\upshape (A)]
\item\label{item:A}
Case without spin-orbit coupling: $\alpha=0$, $\Sigma=\beta\geq0$.
\item\label{item:B}
Case with small spin-orbit coupling: $0<\alpha<\sqrt{2\beta}$ small, $\Sigma=\beta>0$.
\item\label{item:C}
Case with large spin-orbit coupling: $0<\sqrt{2\beta}\leq\alpha$, $0<\beta\leq\Sigma\leq1$,
$E\geq-\Sigma$.
\end{enumerate}
Although the $Q$-function itself is valid for all $\alpha,\beta\geq0$,
the limitation on the parameters is caused by the series representation for the
free Green function (\ref{eq:frg}). In fact, the validity of (\ref{eq:Phizorth}) is governed by (\ref{item:a})--(\ref{item:c}).
In particular, putting $z=i$ in (\ref{eq:Phizorth}), the condition~$\Sigma\leq1$ follows from
(\ref{item:a})--(\ref{item:c}); this latter condition also ensures the existence of the normalization constant
$\mathcal{N}_{s}>0$.

For $\alpha$ zero or arbitrarily small (Cases (\ref{item:A}) or (\ref{item:B}), respectively), 
conditions~(\ref{item:a})--(\ref{item:c}) can be safely removed without affecting the absolute convergence of the series 
that defines the free Green function. This follows from the fact that for $\alpha>0$ small, the confluent Horn 
function, which defines the functions $G_{j}$ in (\ref{eq:frg}), can be analytically continued to the
region $\beta\geq1$; see \cite{Jursenas14} for more details.
As a result one writes $\Sigma\geq0$ in (\ref{item:A}) and $\Sigma>0$ in (\ref{item:B}).
For $E\geq-\Sigma$ in (\ref{item:C}), the eigenfunctions in Lemma~\ref{lem:eigenf}
are restricted to (\ref{item:a}) or (\ref{item:c}).
\subsection{Case without spin-orbit coupling}
Using the asymptotic formula
\begin{equation}
\left(\frac{\alpha}{2}-\frac{ s\beta}{\alpha}\right)
\mrm{artanh}(\alpha\xi(E))=
- s\beta\xi(E)+
\alpha^{2}\xi(E)\left(\frac{1}{2}-
\frac{ s\beta\xi(E)^{2}}{3}\right)+O(\alpha^{4})
\label{eq:sr}
\end{equation}
as $\alpha\downarrow0$, and picking the first term from the right hand side of (\ref{eq:sr}),
we rewrite (\ref{eq:sing}) as follows:
\begin{equation}
\gamma=\prod_{ s}
\left(\omega_{ s}\pm\frac{1}{2\xi(E)}
+ s\beta\xi(E)\right).
\label{eq:singA}
\end{equation}
It takes little effort to conclude from (\ref{eq:xif}) and (\ref{eq:singA}) the following result.
\begin{thm}\label{thm:without}
Let $\alpha=0$ and $0\leq\beta<\infty$. The part of the singular spectrum of $H(\Gamma)$ below the
threshold $-\beta$ is given by the set
$\{E<-\beta\co \gamma=\prod_{s}(\omega_{s}+\sqrt{s\beta-E})\}$. For $\beta>0$, the singular spectrum
embedded in the continuous spectrum is described by the following singletons: $\{-\beta\}$ iff
$\gamma=(\omega_{+}+\sqrt{2\beta})\omega_{-}$; $\{\beta-\omega_{+}^{2}\}$ iff
$\gamma=0$ and $-\sqrt{2\beta}<\omega_{+}<0$; $\{\beta\}$ iff $\gamma=\omega_{-}=0$.
For $\beta=0$, there are no embedded eigenvalues.
\end{thm}
It follows that $H(\Gamma)$ may have maximum three eigenvalues in the interval $[-\beta,\beta]$.
With regard to \cite{Cacciapuoti09}, assume that $\omega_{+}=\omega_{-}\equiv\omega$.
If $\gamma=0$, then Theorem~2 in \cite{Cacciapuoti09} says that the point spectrum is empty for 
$\omega\geq0$, while it follows from Theorem~\ref{thm:without} that the point spectrum consists of the
points $\pm\beta$ when $\omega=0$. If $-\sqrt{2\beta}<\omega<0$ and $\gamma>0$, then Theorem~3 in 
\cite{Cacciapuoti09} says that \textit{there exists} $\gamma_{0}>0$ such that, for $0<\gamma<\gamma_{0}$,
$H(\Gamma)$ does not have embedded eigenvalues. By Theorem~\ref{thm:without}, however, 
there are no embedded eigenvalues \textit{for all} $\gamma>0$; recall that $\gamma=\infty$ and 
$\abs{\omega}=\infty$ together suit $H$.
\subsection{Case with small spin-orbit coupling}
Here we study the behavior of eigenvalues of $H(\Gamma)\equiv H(\Gamma,\alpha)$
in the limit $\alpha\downarrow0$. The main conclusion following
from the results posted below is that, for arbitrarily small but nonzero $\alpha$, the singular spectrum is empty above the 
threshold $-\beta$. Our strategy is to expand $\gamma$ and $\omega_{s}$ in (\ref{eq:gammaomega})
as a series with respect to $\alpha\downarrow0$, and then, using the asymptotic formula (\ref{eq:sr}),
to solve (\ref{eq:sing}).

Using (\ref{eq:phinorm*}), we have that 
\begin{equation}
\mathcal{N}_{s}=
\mathcal{N}_{s}^{(0)}-\alpha^{2}
\mathcal{N}_{s}^{(1)}+O(\alpha^{4})
\label{eq:Nseries}
\end{equation}
as $\alpha\downarrow0$, where 
\begin{subequations}\label{eq:Nseries01}
\begin{align}
\mathcal{N}_{s}^{(0)}:=&
2\sqrt[4]{2}\sqrt{\pi}
\sqrt[4]{\sqrt{1+\beta^{2}}+s\beta}, \\
\mathcal{N}_{s}^{(1)}:=&
\frac{\sqrt{\pi}}{6\sqrt[4]{2}\beta^{2}}
\left(3\beta+s\left(1-\sqrt{1+\beta^{2}}\right)\right)
\sqrt[4]{2+\beta^{2}-2\sqrt{1+\beta^{2}}} 
\left(\sqrt{1+\beta^{2}}+s\beta\right)^{\frac{3}{4}}.
\end{align}
\end{subequations}
Using (\ref{eq:G1X10renX10}) and (\ref{eq:gammaomega}),
\begin{equation}
\Lambda_{s}=\Lambda_{s}^{(0)}+
\alpha^{2}\Lambda_{s}^{(1)}+O(\alpha^{4})
\label{eq:LambaSeries}
\end{equation}
as $\alpha\downarrow0$, where 
\begin{subequations}\label{eq:LambdaSeries01}
\begin{align}
\Lambda_{s}^{(0)}:=&
-\frac{1}{8\pi}\left(\sqrt{ \sqrt{1+\beta^{2}}+1 }
+s\sqrt{ \sqrt{1+\beta^{2}}-1 }\right), \\
\Lambda_{s}^{(1)}:=&
\frac{1}{48\pi\beta^{2}}
\left(3\beta-s\left(1-\sqrt{1+\beta^{2}}\right)\right)
\sqrt[4]{2+\beta^{2}-2\sqrt{1+\beta^{2}}}.
\end{align}
\end{subequations}
If we let
\begin{align}
\omega_{s}^{(0)}:=&
4\pi(\tilde{\Gamma}_{ss}^{(0)}+\Lambda_{s}^{(0)}), \quad
\omega_{s}^{(1)}:=4\pi(\eta_{ss}\tilde{\Gamma}_{ss}^{(0)}
+\Lambda_{s}^{(1)}), \nonumber \\
\gamma^{(0)}:=&(4\pi\abs{ \tilde{\Gamma}_{+-}^{(0)} })^{2}, \quad
\eta_{ss^{\prime}}:=
\frac{ \mathcal{N}_{s}^{(1)} }{ \mathcal{N}_{s}^{(0)} }+
\frac{ \mathcal{N}_{s^{\prime}}^{(1)} }{ 
\mathcal{N}_{s^{\prime}}^{(0)} }, \quad 
\tilde{\Gamma}_{ss^{\prime}}^{(0)}:=
\frac{ \Gamma_{ss^{\prime}} }{ 
\mathcal{N}_{s}^{(0)}\mathcal{N}_{s^{\prime}}^{(0)} }
\label{eq:gammaomega0}
\end{align}
($\omega_{s}^{(0)},\omega_{s}^{(1)},\Lambda_{s}^{(0)},\Lambda_{s}^{(1)}\in\R$; 
$\eta_{ss^{\prime}}>0$; $s,s^{\prime}=\pm$; $\gamma^{(0)}\geq0$), then 
\begin{align}
\omega_{s}=&\omega_{s}^{(0)}+\alpha^{2}\omega_{s}^{(1)}+O(\alpha^{4}), 
\label{eq:omegaSeries} \\
\gamma=&\gamma^{(0)}(1+2\alpha^{2}\eta_{+-})+O(\alpha^{4}),
\label{eq:omegaSeriesB} \\
\tilde{\Gamma}_{ss^{\prime}}=&
\tilde{\Gamma}_{ss^{\prime}}^{(0)}
(1+\alpha^{2}\eta_{ss^{\prime}})+O(\alpha^{4}) 
\label{eq:omegaSeriesA}
\end{align}
as $\alpha\downarrow0$ ($s,s^{\prime}=\pm$). It is clear that $\gamma^{(0)}$,
$\omega_{s}^{(0)}$ coincide with $\gamma$, $\omega_{s}$ in
Theorem~\ref{thm:without}. Put (\ref{eq:omegaSeries}), (\ref{eq:omegaSeriesB}), and (\ref{eq:sr}) in 
equation~(\ref{eq:sing}) and get that 
\begin{align}
\gamma^{(0)}=&
\prod_{s}
\left(\omega_{s}^{(0)}\pm\frac{1}{2\xi(E)}
+s\beta\xi(E)\right) \nonumber \\
&+\alpha^{2}\left(\sum_{s}
q_{-s}(E)
\left(\omega_{s}^{(0)}\pm\frac{1}{2\xi(E)}
+s\beta\xi(E)\right)-2\eta_{+-}\gamma^{(0)}\right) 
+O(\alpha^{4})
\label{eq:singB}
\end{align}
as $\alpha\downarrow0$. Here
\begin{equation}
q_{s}(E):=\omega_{s}^{(1)}-\xi(E)(\tfrac{1}{2}-\tfrac{1}{3}s\beta\xi(E)^{2}).
\label{eq:qsigma}
\end{equation}
When $\alpha=0$, (\ref{eq:singB}) reduces to (\ref{eq:singA}). We look for the solutions of
the form $E=E^{(0)}+\epsilon$ for some small $\epsilon=\epsilon(\alpha)$;
$E^{(0)}\in \sing H(\Gamma,0)$ is in the singular spectrum of
$H(\Gamma,0)$, which is described in Theorem~\ref{thm:without}. It follows that $\epsilon(0)=0$.

Let us pick $\Gamma\equiv\Gamma^{\circ}$
such that $-\beta\in\sing H(\Gamma,0)$ and 
$A\Gamma_{++}+B\Gamma_{--}+C=0$, where 
\begin{align*}
A:=&\mathcal{N}_{-}^{(0)\,2}
(\Lambda_{-}^{(1)}-\eta_{--}\Lambda_{-}^{(0)}), \\
B:=&\mathcal{N}_{+}^{(0)\,2}
\left(\Lambda_{+}^{(1)}-\eta_{++}\left(\Lambda_{+}^{(0)}+\frac{\sqrt{2\beta}}{4\pi}\right)\right), \\
C:=&(\mathcal{N}_{+}^{(0)}\mathcal{N}_{-}^{(0)})^{2}
\left(\Lambda_{-}^{(0)}\Lambda_{+}^{(1)}+\left(\Lambda_{+}^{(0)}+\frac{\sqrt{2\beta}}{4\pi}\right)
\left(\Lambda_{-}^{(1)}-(\eta_{++}+\eta_{--})\Lambda_{-}^{(0)}\right)\right).
\end{align*}
\begin{lem}\label{lem:existence}
Let $0<\alpha<\sqrt{2\beta}$ be arbitrarily small and $\beta<\infty$. 
The following holds:
$-\beta\in\sing H(\Gamma)$ iff $\Gamma=\Gamma^{\circ}$;
$-\beta+\epsilon\notin\sing H(\Gamma)$ $(\epsilon\in\R\backslash\{0\}\;\text{small})$;
$\beta+\epsilon\notin\sing H(\Gamma)$ $(\epsilon\in\R\;\text{small})$.
\end{lem} 
\begin{proof}
By (\ref{eq:xif}), $\xi(-\beta)=1/\sqrt{2\beta}$. Put $E=-\beta$ in (\ref{eq:singB}) 
and get a zero-valued quadratic polynomial in $\alpha>0$. The monomial of
degree $0$ is precisely the condition in Theorem~\ref{thm:without} ensuring 
that $-\beta\in\sing H(\Gamma,0)$. The monomial of degree $1$ is absent, and the monomial of degree $2$ reads
\begin{equation}
2\eta_{+-}\gamma^{(0)}=
q_{-}(-\beta)(\omega_{+}^{(0)}+\sqrt{2\beta})+q_{+}(-\beta)\omega_{-}^{(0)},\quad
q_{s}(-\beta)=\omega_{s}^{(1)}-\frac{1-s/3}{2\sqrt{2\beta}}
\label{eq:deg1}
\end{equation}
($s=\pm$). Using (\ref{eq:gammaomega0}),
\begin{equation}
\omega_{s}^{(1)}=\eta_{ss}\omega_{s}^{(0)}+4\pi(\Lambda_{s}^{(1)}-\eta_{ss}\Lambda_{s}^{(0)}).
\label{eq:sssss2}
\end{equation}
Let $x:=\omega_{+}^{(0)}+\sqrt{2\beta}$, $y:=\omega_{-}^{(0)}$. It follows from 
(\ref{eq:deg1}) and (\ref{eq:sssss2}) that
\begin{align}
2\eta_{+-}\gamma^{(0)}=&(\eta_{++}+\eta_{--})xy-4\pi(\eta_{--}\Lambda_{-}^{(0)}-\Lambda_{-}^{(1)})x
\nonumber \\
&-4\pi\left(\eta_{++}\left(\Lambda_{+}^{(0)}+
\frac{\sqrt{2\beta}}{4\pi}\right)-\Lambda_{+}^{(1)}\right)y.
\label{eq:sssss4}
\end{align}
Since $\eta_{++}+\eta_{--}=2\eta_{+-}$ and $-\beta\in\sing H(\Gamma,0)\Rightarrow\gamma^{(0)}=xy$, 
it follows from (\ref{eq:sssss4}) that $\mathcal{N}_{+}^{(0)\,2}Ax+\mathcal{N}_{-}^{(0)\,2}By=0$, which is
equivalent to $A\Gamma_{++}+B\Gamma_{--}+C=0$. This proves the first statement of the lemma.

Let $E=-\beta+\epsilon$, $\epsilon>0$ small. By (\ref{eq:xif}) and (\ref{eq:qsigma}),
\[\xi(E)=\frac{1}{\sqrt{2\beta}}\left(1-i\sqrt{\frac{\epsilon}{2\beta}}\right)+O(\epsilon),\quad
q_{s}(E)=q_{s}(-\beta)+i\sqrt{\epsilon}\frac{1- s}{4\beta}+O(\epsilon)\]
($s=\pm$). Put these expressions in (\ref{eq:singB}) and get that 
\begin{equation}
\gamma^{(0)}=0, \quad \omega_{+}^{(0)}=-\sqrt{2\beta}, \quad q_{+}(-\beta)=0. 
\label{eq:cnd}
\end{equation}
Then, by the second equality in
(\ref{eq:deg1}), $\omega_{+}^{(1)}=(3\sqrt{2\beta})^{-1}$. Using the latter and (\ref{eq:sssss2}),
\begin{equation}
4\pi(\Lambda_{+}^{(1)}-\eta_{++}\Lambda_{+}^{(0)})-(3\sqrt{2\beta})^{-1}-\eta_{++}\sqrt{2\beta}=0,
\label{eq:cnd0}
\end{equation}
which is false, since by (\ref{eq:Nseries01}), (\ref{eq:LambdaSeries01}), (\ref{eq:gammaomega0}),
the function on the left has a single maximum $\approx-0.14874$ at $\beta\approx1.00553$.
When $\epsilon<0$, (\ref{eq:singB}) also leads to (\ref{eq:cnd}), hence $-\beta+\epsilon\notin\sing H(\Gamma)$
for small nonzero $\epsilon$.

Next $E=\beta+\epsilon$ implies that $\gamma^{(0)}=\omega_{-}^{(0)}=0$ (Theorem~\ref{thm:without}).
When $\epsilon\geq0$, we have that 
\[\xi(E)=\frac{i}{\sqrt{2\beta}}\left(1-\sqrt{\frac{\epsilon}{2\beta}}\right)+O(\epsilon)\]
and
\[q_{s}(E)=q_{s}(\beta)+i\sqrt{\epsilon}\frac{1+s}{4\beta}+O(\epsilon),\quad
q_{s}(\beta)=\omega_{s}^{(1)}-i\frac{1+s/3}{2\sqrt{2\beta}} \]
($s=\pm$). Put these expressions in (\ref{eq:singB}) and deduce that (\ref{eq:singB}) fails for
$\epsilon>0$. When $\epsilon=0$, we obtain the system
$3\omega_{-}^{(1)}\omega_{+}^{(0)}=-1$ and $6\beta\omega_{-}^{(1)}=\omega_{+}^{(0)}$.
The system does not have solutions $\omega_{+}^{(0)}$, $\omega_{-}^{(1)}$ in $\R$, hence 
$\beta+\epsilon\notin\sing H(\Gamma)$ for $\epsilon\geq0$ small. Applying the above procedure
for $\epsilon<0$ we arrive at the same conclusion.
\end{proof}
As previously, let $E=E^{(0)}+\epsilon(\alpha)$. By Lemma~\ref{lem:existence},
$E^{(0)}\neq\pm\beta$ unless $\alpha=0$. We also conclude from Theorem~\ref{thm:without}
and Lemma~\ref{lem:existence} that $E<\beta$. Next we look for $\epsilon(\alpha)$ of the form
$\alpha E^{(1)}+\alpha^{2}E^{(2)}+\alpha^{3}E^{(3)}$ for some
$E^{(j)}\in\R$ ($j=1,2,3$). Put this $E$ in (\ref{eq:xif}):
\begin{equation}
\xi(E)=\sum_{j=0}^{3}\alpha^{j}\xi_{j}+O(\alpha^{4})
\label{eq:xiSeries}
\end{equation}
as $\alpha\downarrow0$, where $\xi_{0}:=\xi(E^{(0)})$ and 
\begin{subequations}\label{eq:xiSeries0123}
\begin{align}
\xi_{1}:=&
-\frac{ \xi_{0}E^{(1)} }{ 2(E^{(0)}+2\beta^{2}\xi_{0}^{2}) }, 
\label{eq:xiSeries0123-1} \\
\xi_{2}:=&
\frac{\xi_{0}}{8}\frac{ E^{(1)\,2}(2\xi_{0}^{2}E^{(0)}-1)
+8\xi_{0}^{2}E^{(2)}
(\beta^{2}-E^{(0)\,2}) }{ (2\xi_{0}^{2}E^{(0)}+1)
(\beta^{2}-E^{(0)\,2}) }, 
\label{eq:xiSeries0123-2} \\
\xi_{3}:=&
-\frac{1}{ 32\xi_{0}(\beta^{2}-E^{(0)\,2})^{3} }\bigl\{
8\beta^{4}[E^{(3)}+\xi_{0}^{2}
(E^{(1)}E^{(2)}+2E^{(0)}E^{(3)})] \nonumber \\
&-E^{(0)\,2}
[8E^{(0)}E^{(1)}E^{(2)}
(1+\xi_{0}^{2}E^{(0)})-
8E^{(0)\,2}E^{(3)}
(1+2\xi_{0}^{2}E^{(0)}) \nonumber \\
&-E^{(1)\,3}(3+2\xi_{0}^{2}E^{(0)})]
+\beta^{2}[8E^{(0)}E^{(1)}E^{(2)} 
-16E^{(0)\,2}E^{(3)}(1+2\xi_{0}^{2}E^{(0)}) \nonumber \\
&+E^{(1)\,3}(1+6\xi_{0}^{2}E^{(0)})]\bigr\}.
\label{eq:xiSeries0123-3}
\end{align}
\end{subequations}
Notice that each denominator in (\ref{eq:xiSeries0123}) is nonzero since $E^{(0)}\neq\pm\beta$. 
Put (\ref{eq:xiSeries}) in (\ref{eq:qsigma}) and get that
\begin{equation}
q_{s}(E)=\sum_{j=0}^{3}\alpha^{j}q_{s}^{(j)}+O(\alpha^{4})\quad(s=\pm)
\label{eq:qSeries}
\end{equation}
as $\alpha\downarrow0$, where $q_{s}^{(0)}:=q_{s}(E^{(0)})$ and
\begin{subequations}\label{eq:qSeries0123}
\begin{align}
q_{s}^{(1)}:=&
-\xi_{1}(\tfrac{1}{2}-s\beta\xi_{0}^{2}), 
\label{eq:qSeries0123-1} \\
q_{s}^{(2)}:=&
-\xi_{2}(\tfrac{1}{2}-s\beta\xi_{0}^{2})+s\beta\xi_{0}\xi_{1}^{2}, 
\label{eq:qSeries0123-2} \\
q_{s}^{(3)}:=&
-\xi_{3}(\tfrac{1}{2}-s\beta\xi_{0}^{2})+\tfrac{1}{3}s\beta\xi_{1}(\xi_{1}^{2}+6\xi_{0}\xi_{2}).
\label{eq:qSeries0123-3}
\end{align}
\end{subequations}
Put (\ref{eq:xiSeries}) and (\ref{eq:qSeries}) in (\ref{eq:singB}), and then collect
the terms with the same powers $<4$ of $\alpha$. Since a cubic polynomial in $\alpha>0$
is zero, set each monomial to zero: A zero-valued monomial of degree $0$ coincides with (\ref{eq:singA}),
and the rest zero-valued monomials of degree $1,2,3$ are of the form
\begin{subequations}\label{eq:systemB}
\begin{align}
0=&
\xi_{1}\sum_{s}\left(\frac{1}{2\xi_{0}^{2}}+ s\beta\right)
\left(\omega_{ s}^{(0)}+\frac{1}{2\xi_{0}}+ s\beta\xi_{0}\right), 
\label{eq:systemB-1} \\
0=&
-2\eta_{+-}\gamma^{(0)}+
\xi_{1}^{2}\prod_{s}\left(\frac{1}{2\xi_{0}^{2}}+ s\beta\right) \nonumber \\
&+\sum_{s}\left(\omega_{ s}^{(0)}+\frac{1}{2\xi_{0}}+ s\beta\xi_{0}\right)
\left(q_{- s}^{(0)}+\frac{ \xi_{1}^{2}-\xi_{0}\xi_{2} }{ 2\xi_{0}^{3} }- s\beta\xi_{2}\right), 
\label{eq:systemB-2} \\
0=&
\sum_{s}\left(\xi_{1}\left(-\frac{1}{2\xi_{0}^{2}}+ s\beta\right)
\left(q_{- s}^{(0)}+\frac{ \xi_{1}^{2}-\xi_{0}\xi_{2} }{ 2\xi_{0}^{3} }
- s\beta\xi_{2}\right)\right. \nonumber \\
&\left.+\left(\omega_{ s}^{(0)}+\frac{1}{2\xi_{0}}+ s\beta\xi_{0}\right)
\left(q_{- s}^{(1)}-\frac{ \xi_{1}^{3}-2\xi_{0}\xi_{1}\xi_{2}+\xi_{0}^{2}\xi_{3} }{ 2\xi_{0}^{4} }
- s\beta\xi_{3}\right)\right). 
\label{eq:systemB-3}
\end{align}
\end{subequations}
Equations~(\ref{eq:systemB-1}) and (\ref{eq:systemB-2}) imply that
\begin{prop}\label{prop:lambda1}
$E^{(1)}=0$.
\end{prop}
\begin{proof}
Let $E^{(0)}<-\beta$. Then (\ref{eq:systemB-1}) simplifies thus:
\[0=\xi_{1}\sum_{ s}\sqrt{  s\beta-E^{(0)} }
\left(\sqrt{  s\beta-E^{(0)} }+\sqrt{ - s\beta-E^{(0)} }\right)
\left(\omega_{ s}^{(0)}+\sqrt{  s\beta-E^{(0)} }\right).\]
Since $s\beta-E^{(0)}>0$ ($s=\pm$), it follows that
\[\sqrt{  s\beta-E^{(0)} }
\left(\sqrt{  s\beta-E^{(0)} }+\sqrt{ - s\beta-E^{(0)} }\right)>0\]
for both $s=\pm$. But then, either $\xi_{1}=0$ or $\omega_{ s}^{(0)}+\sqrt{  s\beta-E^{(0)} }=0$.
If the former, then, by (\ref{eq:xiSeries0123-1}), $E^{(1)}=0$. If the latter, then 
\begin{equation}
E^{(0)}=\beta-\omega_{+}^{(0)\,2}, \quad \gamma^{(0)}=0, \quad 
\omega_{+}^{(0)}<-\sqrt{2\beta}, \quad \omega_{-}^{(0)}=-\sqrt{\omega_{+}^{(0)\,2}-2\beta}. 
\label{eq:cnd2}
\end{equation}
Under these conditions, (\ref{eq:systemB-2}) fails
unless $\xi_{1}=0$. Therefore, $E^{(1)}=0$ in either case.

If $-\beta<E^{(0)}\leq0$, then $-\sqrt{2\beta}<\omega_{+}^{(0)}\leq-\sqrt{\beta}$. But then
(\ref{eq:systemB-1}) implies either $\xi_{1}=0$ or $\omega_{+}^{(0)}=\omega_{-}^{(0)}$ and $\beta=0$.
Since $\beta>0$, $E^{(1)}=0$ necessarily. Finally, when $0<E^{(0)}<\beta$, (\ref{eq:systemB-1})
also implies that $\xi_{1}=0$ necessarily.
\end{proof}
By (\ref{eq:qSeries0123-1}) and Proposition~\ref{prop:lambda1}, $q_{s}^{(1)}=0$, and the system (\ref{eq:systemB})
reduces to (\ref{eq:systemB-2})--(\ref{eq:systemB-3}). It follows that $\xi_{2}$ and $\xi_{3}$ are uniquely defined
by the latter system unless (\ref{eq:cnd2}) holds. Otherwise we need more terms of expansion. Thus,  
$O(\alpha^{4})$ in (\ref{eq:sr}) further reads
\begin{equation}
O(\alpha^{4})=\alpha^{4}\xi(E)^{3}(\tfrac{1}{6}
-\tfrac{1}{5}s\beta\xi(E)^{2})
+\alpha^{6}\xi(E)^{5}(\tfrac{1}{10}
-\tfrac{1}{7}s\beta\xi(E)^{2})+O(\alpha^{8}).
\label{eq:sr2}
\end{equation}
By definition, the normalization constant $\mathcal{N}_{s}$ depends on even powers of $\alpha$. Then
$O(\alpha^{4})$ in (\ref{eq:omegaSeries}) is given by
\begin{equation}
O(\alpha^{4})=\alpha^{4}\omega_{ s}^{(2)}+\alpha^{6}\omega_{ s}^{(3)}+
O(\alpha^{8}).
\label{eq:omegaSeries2}
\end{equation}
Although $\omega_{ s}^{(2)}$ and $\omega_{ s}^{(3)}$
can be found similar to $\omega_{s}^{(1)}$ in (\ref{eq:gammaomega0}),
their explicit representation is inessential since our final goal is to express $E$
with the accuracy up to $O(\alpha^{4})$.
Next, if we supplement $\epsilon(\alpha)$ with terms
indexed by $j=4,5,6$, then the values of functions $\xi$ and $q_{s}$ 
in (\ref{eq:xiSeries}) and (\ref{eq:qSeries}), respectively, are also supplemented
with terms indexed by $j=4,5,6$. Put these values along with (\ref{eq:sr2}) and
(\ref{eq:omegaSeries2}) in (\ref{eq:sr}), (\ref{eq:omegaSeries}), and then in
(\ref{eq:sing}). Then collect the terms with the same powers $<7$ of $\alpha$,
and deduce a zero-valued polynomial in $\alpha>0$ of degree $6$.
Set each monomial to zero, apply Proposition~\ref{prop:lambda1}, the relation $q_{s}^{(1)}=0$, 
and condition (\ref{eq:cnd2}), and get that the zero-valued monomials of degree 
$4,5,6$ are of the form: 
\begin{subequations}\label{eq:systemB2}
\begin{align}
0=&\prod_{s}\left(q_{-s}^{(0)}-\xi_{2}\left(
\frac{1}{2\xi_{0}^{2}}+ s\beta\right)\right), \label{eq:systemB2-1} \\
0=&\xi_{3}\left(2\xi_{2}\prod_{s}\left(
\frac{1}{2\xi_{0}^{2}}+ s\beta\right)-\sum_{s}q_{ s}^{(0)}
\left(\frac{1}{2\xi_{0}^{2}}+ s\beta\right)\right), \label{eq:systemB2-2} \\
0=&30\xi_{3}^{2}\prod_{s}\left(\frac{1}{2\xi_{0}^{2}}+ s\beta\right)
-\sum_{s}\left(q_{-s}^{(0)}-\xi_{2}\left(
\frac{1}{2\xi_{0}^{2}}+ s\beta\right)\right) \nonumber \\
&\times\left(\xi_{0}^{5}\left(\frac{5}{\xi_{0}^{2}}-6 s\beta\right)+
30\left(\xi_{0}^{2}\xi_{2}+\xi_{4}\right)\left(\frac{1}{2\xi_{0}^{2}}- s\beta\right)
-\frac{15\xi_{2}^{2}}{\xi_{0}^{3}}-30\omega_{ s}^{(2)}\right).
\label{eq:systemB2-3}
\end{align}
\end{subequations}
\begin{prop}\label{prop:lambda3}
$E^{(3)}=0$.
\end{prop}
\begin{proof}
By Proposition~\ref{prop:lambda1} and (\ref{eq:xiSeries0123-3}),
\[\xi_{3}=\frac{E^{(3)}\left(2\xi_{0}^{2}E^{(0)}+1\right)}{4\xi_{0}
\left(E^{(0)\,2}-\beta^{2}\right)}.\]
Therefore, $\xi_{3}=0\Leftrightarrow E^{(3)}=0$. Since $q_{s}^{(1)}=0$ and $\xi_{1}=0$,
we see that (\ref{eq:systemB-3}) coincides with (\ref{eq:systemB-1}) but with $\xi_{1}$ replaced by $\xi_{3}$.
If (\ref{eq:cnd2}) does not hold, then we can redo all the steps made up
during the proof of Proposition~\ref{prop:lambda1} to conclude that $\xi_{3}=0$.
Hence $E^{(3)}=0$, as claimed. If (\ref{eq:cnd2}) holds, then by (\ref{eq:systemB2-2}), either
$\xi_{3}=0$ or 
\begin{equation}
\xi_{2}=\tfrac{1}{2}\sum_{ s}q_{ s}^{(0)}
\left(\frac{1}{2\xi_{0}^{2}}+ s\beta\right)/\prod_{ s}\left(
\frac{1}{2\xi_{0}^{2}}+ s\beta\right).
\label{eq:cnd3}
\end{equation}
Substitute (\ref{eq:cnd3}) in (\ref{eq:systemB2-1}) and get that $\xi_{0}>0$
takes the value
\begin{equation}
\xi_{0}=\sqrt{\frac{1}{2\beta}\frac{ q_{-}^{(0)}-q_{+}^{(0)} }{ q_{-}^{(0)}+q_{+}^{(0)} }}\quad
\text{for}\quad \frac{ q_{-}^{(0)}-q_{+}^{(0)} }{ q_{-}^{(0)}+q_{+}^{(0)} }>0.
\label{eq:ppppp}
\end{equation}
Put this value in (\ref{eq:cnd3}) and get that $\xi_{2}=(q_{-}^{(0)}-q_{+}^{(0)})/(2\beta)$. Now put this $\xi_{2}$ and 
(\ref{eq:ppppp}) in (\ref{eq:systemB2-3}) and deduce that either $\xi_{3}=0$ or $q_{+}^{(0)}q_{-}^{(0)}=0$. If
the latter, then $q_{+}^{(0)}=0$, since the inequality in (\ref{eq:ppppp}) fails for $q_{-}^{(0)}=0$. If
$q_{+}^{(0)}=0$, then $\omega_{-}^{(0)}(\omega_{+}^{(0)}-\omega_{-}^{(0)})=0$. But $\omega_{-}^{(0)}<0$
and therefore $\omega_{+}^{(0)}=\omega_{-}^{(0)}$, which is true iff $\beta=0$, hence false.
\end{proof}
\begin{prop}\label{prop:lambda2}
For $E^{(0)}<-\beta$,
\begin{align*}
E^{(2)}=&
\frac{2\sqrt{2}}{\beta}(\sqrt{ E^{(0)\,2}-\beta^{2} }-E^{(0)})
\sqrt{ E^{(0)\,2}-\beta^{2} }\sqrt{-E^{(0)}-\sqrt{ E^{(0)\,2}-\beta^{2} }}
\nonumber \\
&\times
\frac{ \sum_{ s}q_{- s}^{(0)}(\omega_{ s}^{(0)}+
\sqrt{ s\beta-E^{(0)}})-2\eta_{+-}\gamma^{(0)} }{ 
\sum_{ s}\sqrt{ s\beta-E^{(0)}}(\sqrt{ s\beta-E^{(0)}}+
\sqrt{- s\beta-E^{(0)}})(\omega_{ s}^{(0)}+
\sqrt{ s\beta-E^{(0)}}) } \\
&\qquad\qquad\qquad(\gamma^{(0)}\neq0); \\
=&-2\omega_{+}^{(0)}q_{+}^{(0)}\qquad(\gamma^{(0)}=0,\,\omega_{+}^{(0)}<-\sqrt{2\beta},
\,\omega_{-}^{(0)}\neq-\sqrt{\omega_{+}^{(0)\,2}-2\beta}); \\
=&-2\omega_{-}^{(0)}q_{-}^{(0)}\qquad(\gamma^{(0)}=0,\,\omega_{-}^{(0)}<0,\,
\omega_{+}^{(0)}\neq-\sqrt{\omega_{-}^{(0)\,2}+2\beta}).
\end{align*}
Further, if $\gamma^{(0)}=0$, $\omega_{+}^{(0)}<-\sqrt{2\beta}$, and
$\omega_{-}^{(0)}=-\sqrt{\omega_{+}^{(0)\,2}-2\beta}$, then $E^{(2)}$
is twofold, and it takes the values $-2\omega_{ s}^{(0)}q_{ s}^{(0)}$ $( s=\pm)$.
\end{prop}
\begin{rem}
It can be shown that, for $\Gamma$ diagonal and $E^{(0)}<-\beta$, the following holds:
$q_{-}^{(0)}<0$ for $E^{(0)}=-\beta-\omega_{-}^{(0)\,2}$, $\omega_{-}^{(0)}<0$; 
$q_{+}^{(0)}<0$ for $E^{(0)}=\beta-\omega_{+}^{(0)\,2}$, $\omega_{+}^{(0)}<-\sqrt{2\beta}$. 
By Proposition~\ref{prop:lambda2}, this means the spin-orbit term moves the eigenvalues down from the threshold 
$-\beta$.
\end{rem}
\begin{proof}[Proof of Proposition~\ref{prop:lambda2}]
Using Proposition~\ref{prop:lambda1}, one finds from (\ref{eq:xiSeries0123-2}) that
\begin{equation}
\xi_{2}=\frac{E^{(2)}\xi_{0}^{3}}{2\xi_{0}^{2}E^{(0)}+1}.
\label{eq:aaaaaa-1}
\end{equation}
When (\ref{eq:cnd2}) does not hold, $E^{(2)}$ is found from (\ref{eq:systemB-2}) and (\ref{eq:aaaaaa-1}). 
If (\ref{eq:cnd2}) holds, then $E^{(2)}$ follows from (\ref{eq:systemB2-1}) and (\ref{eq:aaaaaa-1}). 
\end{proof}
\begin{thm}\label{thm:existence}
Let $0<\alpha<\sqrt{2\beta}$ be arbitrarily small and $\beta<\infty$.
Then the part of the singular spectrum of $H(\Gamma,\alpha)$ that is below the threshold $-\beta$ consists of the points
$E^{(0)}+\alpha^{2}E^{(2)}+O(\alpha^{4})$ where $E^{(0)}<-\beta$ belongs to the singular 
spectrum of $H(\Gamma,0)$. Let $E\geq-\beta$ be in the singular spectrum of $H(\Gamma,\alpha)$. Then
$E=-\beta$ and $\Gamma=\Gamma^{\circ}$ necessarily; otherwise $E$ does not belong to the
singular spectrum.
\end{thm}
\begin{proof}
By Propositions~\ref{prop:lambda1}, \ref{prop:lambda3}, \ref{prop:lambda2}, the singular spectrum of 
$H(\Gamma,\alpha)$ consists of the points $E=E^{(0)}+\alpha^{2}E^{(2)}$ up to $O(\alpha^{4})$. 
To prove the first part of the theorem, we need only show $E^{(0)}<-\beta\Rightarrow E<-\beta$. 
Let $E^{(0)}=-\beta-\delta$,
$\delta>0$ small. Then $E=-\beta+\epsilon$ where $\epsilon:=\alpha^{2}E^{(2)}-\delta$.
By Proposition~\ref{prop:lambda2}, $E^{(2)}=c\sqrt{\delta}+O(\delta)$ for some real $c$ whose explicit 
representation is inessential here. But then $\epsilon<0\Leftrightarrow\delta>O(\alpha^{4})$, as claimed.
To prove the second part, we need to show $E^{(0)}\in(-\beta,\beta)\Rightarrow E\notin \sing 
H(\Gamma,\alpha)$, as it follows from Lemma~\ref{lem:existence}. By Theorem~\ref{thm:without} and 
(\ref{eq:systemB-2}),
\begin{equation}
q_{+}^{(0)}=\xi_{2}\left(\frac{1}{2\xi_{0}^{2}}-\beta\right)
\label{eq:empty}
\end{equation}
and we show that equation~(\ref{eq:empty}) does not have solutions with respect to $\xi_{2}\in \R$.
Let $-\beta<E^{(0)}<0$. The real and imaginary parts of (\ref{eq:empty}) imply that
\begin{subequations}\label{eq:l2}
\begin{align}
E^{(2)}=&
-\frac{\omega_{+}^{(0)}}{3\beta^{2}}
(6\beta^{2}\omega_{+}^{(1)}+\omega_{+}^{(0)}
(5\beta-2\omega_{+}^{(0)\,2})) \label{eq:l2-a} \\
=&\frac{1}{3\beta^{2}}
(2\beta^{2}(2-3\omega_{+}^{(0)}\omega_{+}^{(1)})
-\omega_{+}^{(0)}(7\beta-2\omega_{+}^{(0)\,2})).
\label{eq:l2-b}
\end{align}
\end{subequations}
It follows that (\ref{eq:l2-a}) $=$ (\ref{eq:l2-b}) iff $2\beta=\omega_{+}^{(0)\,2}$,
which contradicts the initial hypothesis $-\sqrt{2\beta}<\omega_{+}^{(0)}<0$ (Theorem~\ref{thm:without}).
If $E^{(0)}=0$, then $E^{(2)}=-\frac{1}{3}$ and $\omega_{+}^{(1)}=\frac{1}{3}\beta^{-1/2}$.
Put this $\omega_{+}^{(1)}$ along with $\omega_{+}^{(0)}=-\sqrt{\beta}$ in (\ref{eq:sssss2}) and get 
(\ref{eq:cnd0}) with $\sqrt{2\beta}$ replaced by $\sqrt{\beta}$. In this case the left-hand side approaches 
$-\frac{5}{24}\beta^{-3/2}$ as $\beta\uparrow\infty$, hence false. When $0<E^{(0)}<\beta$,
$E^{(2)}$ satisfies (\ref{eq:l2}), and we conclude that $E^{(0)}+\alpha^{2}E^{(2)}\notin 
\sing H(\Gamma,\alpha)$ for all $E^{(0)}\in(-\beta,\beta)$.
\end{proof}
\subsection{Case with large spin-orbit coupling}
Here we wish to post some most important properties of solutions $E\geq-\Sigma$ of (\ref{eq:sing}) 
under hypothesis in (\ref{item:C}). For this we find it convenient to define the functions $(0,\nu]\to\R$ as
\begin{align*}
U_{\nu}(x):=&\frac{1}{x}\left(\frac{\nu^{2}+1}{\nu^{2}}\arctan(x)-\frac{1}{x}\right), \\
V_{\nu}(x):=&\frac{\nu^{2}}{\nu^{2}+1}U_{\nu}(x)\left(2-\left(\nu^{2}-1\right)x^{2}U_{\nu}(x)\right).
\end{align*}
The parameter $\nu$ is defined as $\nu:=\alpha/\sqrt{2\beta}$, and it takes the values
\[1\leq\nu\leq \sqrt{\frac{1+\sqrt{1-\beta^{2}}}{\beta}} \quad 
(0<\beta\leq1).\]
\small
\begin{figure}
\centering
\includegraphics[width=1.00\textwidth]{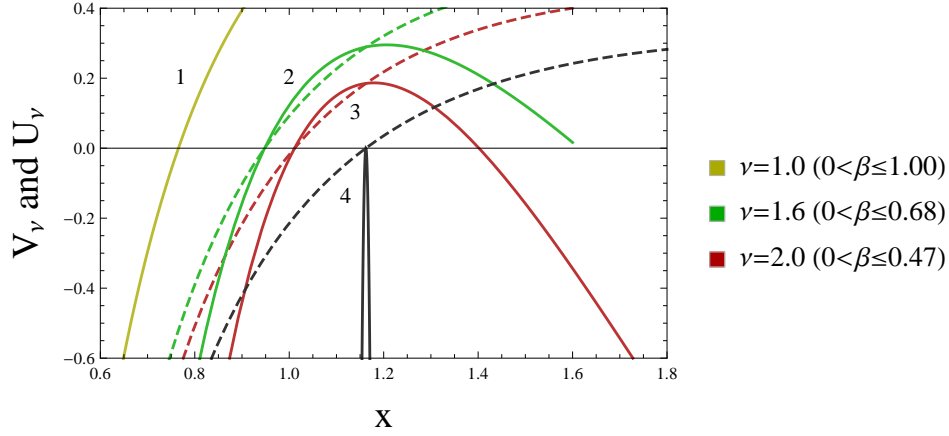}
\caption{The plots of functions $V_{\nu}$ (solid curves) and $U_{\nu}$ (dashed curves)
at $0.6\leq x\leq1.8$. The curve labeled by 1 shows the values of $V_{1.0}=U_{1.0}$;
the (solid and dashed) curves labeled by 2 show the values of $V_{1.6}$ and $U_{1.6}$;
the (solid and dashed) curves labeled by 3 show the values of $V_{2.0}$ and $U_{2.0}$;
the (solid and dashed) curves labeled by 4 show the values of $V_{\infty}$ and $U_{\infty}$,
\ie $V_{\nu}$ and $U_{\nu}$ as $\nu\uparrow\infty$. The function $V_{\infty}$ takes the value $-\infty$
everywhere except at $x_{\infty,1}\approx1.16234$, which is also a zero of $U_{\infty}$.
For $\nu<\infty$, a zero of $U_{\nu}$ coincides with the first zero, $x_{\nu,1}$, of $V_{\nu}$, and
$V_{\nu}$ also has the second zero, $x_{\nu,2}$, if $x_{\nu,2}\leq\nu$. In figure,
$x_{1.6,2}\approx1.61471>1.6$ and $x_{2.0,2}\approx1.4018<2.0$.}
\label{fig:fig}
\end{figure}
\normalsize
The parameter $\nu$ is predetermined by the relations $\alpha\geq\sqrt{2\beta}$ and $\Sigma\leq1$, and by the 
definition of $\Sigma$. Some typical plots of the functions $V_{\nu}$ and $U_{\nu}$ are illustrated in Fig.~\ref{fig:fig}. 
In particular, if $\beta>0$ is arbitrarily small, then $\nu$ varies from $1$ to $\sqrt{2/\beta}$
depending on $\sqrt{2\beta}\leq\alpha\leq2$. Therefore, the curves in Fig.~\ref{fig:fig} labeled by 4 
represent the case when the magnetic field $\beta$ approaches zero from above and the
spin-orbit-coupling strength $\alpha$ is large enough.

The function $V_{\nu}$ has at most two zeros, which we denote by
$x_{\nu,1}$ and $x_{\nu,2}$. We assume that $x_{\nu,1}>0$ is the zero of
$U_{\nu}$. Then the second zero $x_{\nu,2}>0$ solves the equation
$U_{\nu}(x)=2/((\nu^{2}-1)x^{2})$ ($\nu>1$)
provided that $x_{\nu,2}\leq\nu$. The relation $x_{\nu,2}>\nu$ is meaningless
because the domain of $V_{\nu}$ is $(0,\nu]$, by definition. If $\nu=1$, then 
the function $V_{1}=U_{1}$ has only one zero $x_{1,1}\approx0.76538$.
If $\nu\uparrow\infty$, we write $V_{\infty}:=\lim V_{\nu}$, and $-V_{\infty}$
formally behaves like an ill-defined Dirac delta function concentrated at
a zero $x_{\infty,1}\approx1.16234$ of $U_{\infty}$. In this latter case,
the two zeros $x_{\infty,1}$ and $x_{\infty,2}$ coincide. In all other cases,
\ie for $1<\nu<\infty$, the function $V_{\nu}$ has either two zeros,
if $x_{\nu,2}\leq\nu$, or a single zero, if otherwise. In general,
$x_{\nu,1}\leq x_{\nu,2}$, and the equality holds only if $\nu=\infty$.

Yet another quantity which we use to describe our main results is a function
$E_{\nu}\co (0,\nu]\to [\beta,\infty)$ defined as
\[E_{\nu}(x):=\frac{\nu^{4}+x^{4}}{2\left(\nu x\right)^{2}}\beta\quad\text{and}\quad
E_{\nu,n}:=E_{\nu}(x_{\nu,n})\quad(n=1,2).\]
The relations $x_{\nu,1}\leq x_{\nu,2}\leq\nu$ imply that $\beta\leq E_{\nu,2}\leq E_{\nu,1}$.
In particular, we have that $E_{1,1}\approx1.14643\beta$. 
In the limit $\beta\downarrow0$, the function $E_{\nu}(x)$
ranges from $0$ to $x^{-2}$ when $\nu$ ranges from
$1$ to $\sqrt{2/\beta}$. Choosing $\nu=\sqrt{2/\beta}\uparrow\infty$
(\ie $\alpha=2$, $\beta\downarrow0$) one deduces $E_{\infty,1}=E_{\infty,2}\approx0.74018$.
\begin{thm}\label{thm:existence3}
Let $\sqrt{2\beta}\leq\alpha\leq\sqrt{2(1+\sqrt{1-\beta^{2}})}$, $0<\beta\leq1$.
The singular spectrum of $H(\Gamma)$, which is embedded in the continuous spectrum $[-\Sigma,\infty)$, is given by
the set
\begin{align*}
\bigl\{E_{\nu}(x)\in[\beta,E_{\nu,1}]\co & x\in[x_{\nu,1},\nu],\,
\gamma=\omega_{+}\omega_{-}+(\beta/2)V_{\nu}(x), \\
&2\omega_{-}=x^{2}U_{\nu}(x)\sum_{s}\omega_{ s}
\left(\nu^{2}+ s\right)\bigr\}.
\end{align*}
\end{thm}
For example, if $\beta$ is arbitrarily small and $\alpha$ 
large, say $\alpha=2$, then $\nu$ approaches $\infty$ from below, and Theorem~\ref{thm:existence3} tells us that 
there is an eigenvalue, $E_{\infty,1}$($\approx0.74018$), above the threshold $-1$ iff 
$\gamma=\omega_{+}\omega_{-}$. This particular eigenvalue fails to satisfy
(\ref{item:b}) and (\ref{item:c}), so the corresponding eigenfunction in Lemma~\ref{lem:eigenf}
has to be analytically continued applying the methods similar to those used in \cite{Jursenas14}.
\begin{proof}[Proof of Theorem~\ref{thm:existence3}]
Since $\mrm{artanh}(\alpha\xi(-\Sigma))=\infty$, $E=-\Sigma$ does not solve (\ref{eq:sing}).
If $-\Sigma<E\leq-\beta$, then $\alpha\xi(E)>1$, $\alpha>\sqrt{2\beta}$, and 
$\mrm{artanh}(\alpha\xi(E))$ is of the form $r-i\pi/2$, where $r>0$ is the real part.
Put this form in (\ref{eq:sing}) and deduce the system of equations
\begin{equation}
\gamma=\prod_{s}a_{s}-\prod_{s}b_{s}, \quad 0=\sum_{s}a_{s}b_{s}
\label{eq:subsystem}
\end{equation}
where 
\[a_{s}:=\omega_{s}+\frac{1}{2\xi(E)}-
r\left(\frac{\alpha}{2}-\frac{s\beta}{\alpha}\right),\quad
b_{s}:=\frac{\pi}{2}
\left(\frac{\alpha}{2}+\frac{s\beta}{\alpha}\right) \quad (s=\pm).\]
Now that $a_{s}\in\R$ and $b_{s}>0$, it follows that 
$\gamma=-(b_{+}/b_{-})(a_{+}^{2}+b_{+}^{2})<0$, in contradiction to $\gamma\geq0$.
Next, suppose that $-\beta<E<\beta$. In this case $\xi(E)=R+iT$ where the real part $R>0$ while the 
imaginary part $T<0$ for $-\beta<E<0$, and $T>0$ for $0\leq E<\beta$. If we let
\begin{align*}
r:=&\re\mrm{artanh}(\alpha\xi(E))=\tfrac{1}{4}\ln\left(
\frac{ (1+\alpha R)^{2}+(\alpha T)^{2} }{ 
(1-\alpha R)^{2}+(\alpha T)^{2} }\right) \quad (r>0), \\
t:=&\img\mrm{artanh}(\alpha\xi(E))=\tfrac{1}{2}
(\Arg(1+\alpha R+i\alpha T)-\Arg(1+\alpha R-i\alpha T))
\end{align*}
then $t<0$ for $-\beta<E<0$, and $t>0$ for $0\leq E<\beta$. Put $\xi(E)=R+iT$ and 
$\mrm{artanh}(\alpha\xi(E))=r+it$ in (\ref{eq:sing}) and deduce the system (\ref{eq:subsystem}) where now
\[a_{s}:=\omega_{s}+\frac{R}{2\abs{\xi(E)}^{2}}-
r\left(\frac{\alpha}{2}-\frac{s\beta}{\alpha}\right),\quad
b_{s}:=\frac{T}{2\abs{\xi(E)}^{2}}+
t\left(\frac{\alpha}{2}+\frac{s\beta}{\alpha}\right).\]
We see that $a_{s}\in\R$ and $b_{s}<0$ for $-\beta<E<0$, and $b_{s}>0$ for $0\leq E<\beta$.
But then (\ref{eq:subsystem}) leads to the previous $\gamma=-(b_{+}/b_{-})(a_{+}^{2}+b_{+}^{2})<0$.
Finally, if $E\geq\beta$, then $\xi=iT$, where $T(E)>0$ reads
\[T(E):=\frac{1}{\sqrt{2}\beta}
\sqrt{ E-\sqrt{ E^{2}-\beta^{2} } },\quad
0<T(E)\leq T(\beta)=1/\sqrt{2\beta}.\]
Then $\mrm{artanh}(\alpha\xi)=i\arctan(\alpha T)$,
and (\ref{eq:sing}) reduces to (\ref{eq:subsystem}) where
\[a_{s}:=\omega_{s},\quad
b_{s}:=\frac{1}{2T(E)}-\left(\frac{\alpha}{2}+
\frac{s\beta}{\alpha}\right)
\arctan(\alpha T(E))\]
and $a_{s},b_{s}\in\R$. Using the definition of $T$, we write $E=E_{\nu}(x)$ where
$x:=\nu T(E)\sqrt{2\beta}$ ($0<x\leq\nu$). When written in terms of $\nu$, $x$, $U_{\nu}$, $V_{\nu}$,
the system (\ref{eq:subsystem}) coincides with that defining the singular spectrum in the theorem. It remains to
show that $x\geq x_{\nu,1}$, \ie $E_{\nu}(x)\leq E_{\nu,1}$. Let $\nu=1$. Since $V_{1}=U_{1}$, it
follows from (\ref{eq:subsystem}) that $\gamma=U_{1}(x)((\omega_{+}x)^{2}+(\beta/2))\geq0$, hence $U_{1}(x)\geq0$. 
But then $x\in[x_{1,1},1]$ and $E_{1}(x)\in [\beta,E_{1,1}]$. Let $\nu>1$ and $x\in (0,x_{\nu,1})$.
In this case $U_{\nu}(x),V_{\nu}(x)<0$ and $(\nu^{2}-1)x^{2}U_{\nu}(x)<2$. On the other hand, the system 
(\ref{eq:subsystem}) implies that $\omega_{-}=-c\omega_{+}$, $c>0$, and thus
$\gamma=-(\omega_{+}^{2}c+(\beta/2)\abs{V_{\nu}(x)})<0$. We conclude that $x\in [x_{\nu,1},\nu]$.
\end{proof}
\section{Discussion}
In the current paper, the main results concerning the eigenvalues (or else the singular points) are summarized in 
Theorems~\ref{thm:without}, \ref{thm:existence}, \ref{thm:existence3}. Other parts of the spectrum, including
the eigenfunctions, are discussed in Sec.~\ref{sec:sp} and particularly in Lemma~\ref{lem:eigenf} and
Theorem~\ref{thm:sc}. Although we said nothing about the eigenvalues below the threshold $-\Sigma$ when 
$\alpha\geq\sqrt{2\beta}$, their existence can be seen from a general equation~(\ref{eq:sing}). For example, taking 
$\Gamma:=-C^{-1}-R$ diagonal with the entries such that $\omega_{+}=\omega_{-}\equiv\omega$, one finds that
\begin{equation}
\omega+\sqrt{-E}=\frac{\alpha}{2}\mrm{artanh}\left(\frac{\alpha}{2\sqrt{-E}}\right)
\label{eq:last}
\end{equation}
\small
\begin{figure}
\centering
\begin{subfigure}[b]{0.48\textwidth}
\includegraphics[width=\textwidth]{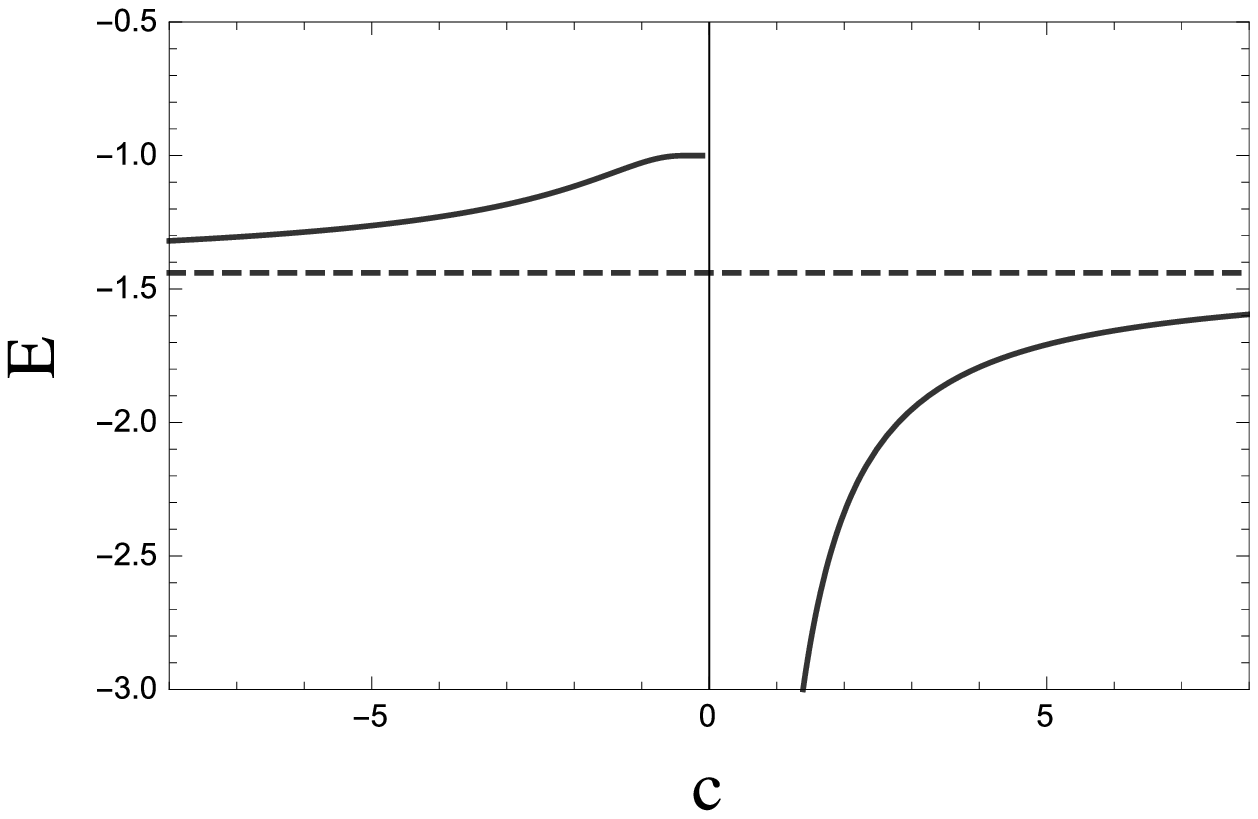}
\caption{$\beta\sim0$}
\label{fig:D}
\end{subfigure}\hfill 
\begin{subfigure}[b]{0.48\textwidth}
\includegraphics[width=\textwidth]{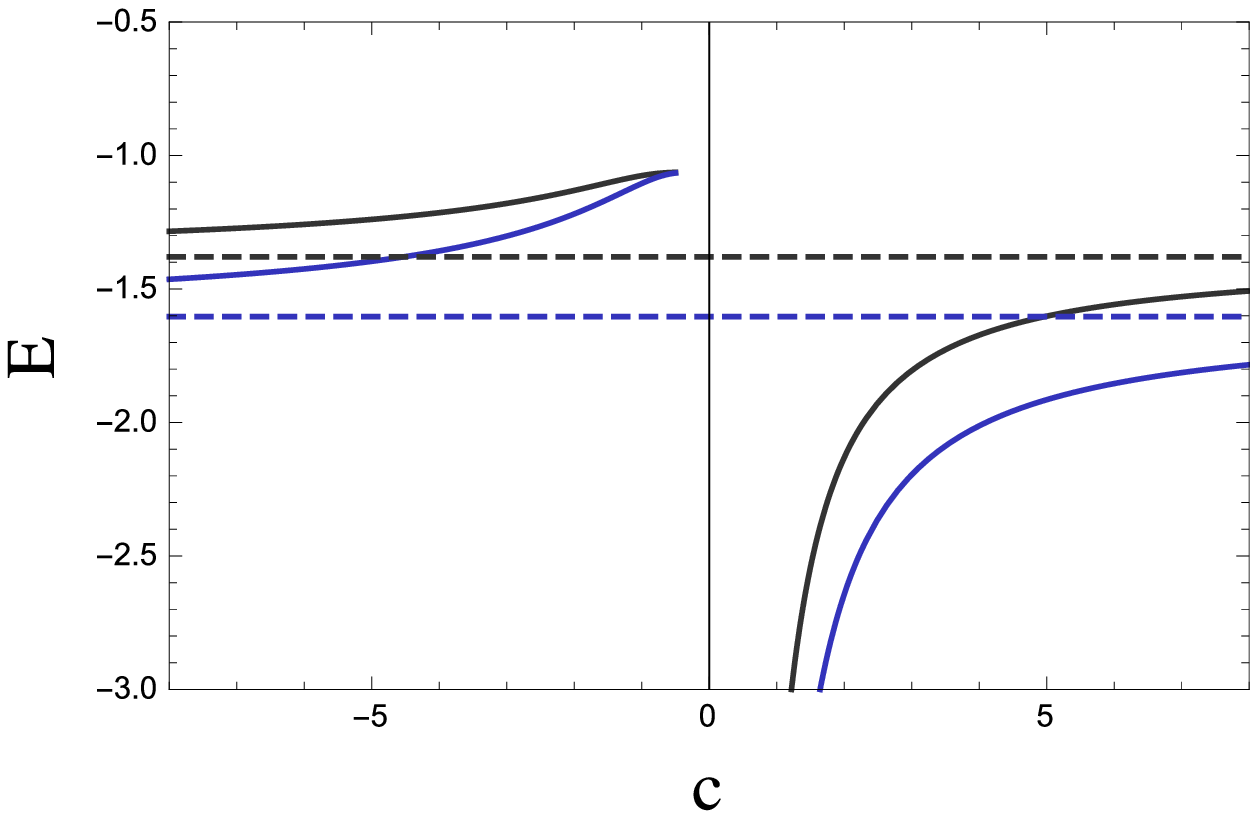}
\caption{$\beta=1/2$}
\label{fig:E}
\end{subfigure}
\caption{The eigenvalue $E$ of the perturbed Hamiltonian (\ref{eq:HV}) as a function of impurity scattering strength 
$C=c\mathbb{I}$ when the admissible matrix $R=r\mathbb{I}$ with $r\approx-0.17850$. 
The spin-orbit-coupling strength $\alpha=2$. When $c=0$, the perturbed Hamiltonian coincides with the free 
Hamiltonian whose spectrum $[-1-\beta^{2}/4,\infty)$ is absolutely continuous. When $c^{-1}=0$, the perturbed 
Hamiltonian is described by the Friedrichs extension whose spectrum is that of the free Hamiltonian.
In (a) the threshold approaches $-1$ and the dashed line shows the eigenvalue $E\approx-1.43923$ corresponding 
to the coupling parameter $c$ whose absolute value is arbitrarily large but finite. In (b) the threshold is 
$-17/16=-1.06250$, the black (upper) solid curves represent the eigenvalue for $s=+$ in (\ref{eq:sing}), the blue 
(lower) solid curves represent the eigenvalue for $s=-$, and the black dashed (upper) and blue dashed (lower) lines show
the eigenvalues $E\approx-1.37956$ (for $s=+$) and $E\approx-1.60313$ (for $s=-$) corresponding to $c$ whose 
absolute value is arbitrarily large but finite.}
\label{fig:fig2}
\end{figure}\normalsize
provided that the magnetic field $\beta$ is arbitrarily close to $0$. In this regime, the eigenvalue $E<-\Sigma$ 
exists if $\Gamma$ is of the form $v\mathbb{I}$ for some real $v$, where $\mathbb{I}$ is the identity matrix; 
this follows from the conditions $\omega_{+}=\omega_{-}$ and $\beta\sim0$ and the definition in 
(\ref{eq:gammaomega}), which in turn gives a one-to-one correspondence between $\omega$ and $v$. Note that, 
when $\alpha=0$, (\ref{eq:last}) gives the solution $E=-\omega^{2}$, $\omega<0$, which is (by no accident) in exact 
agreement with that described in Theorem~\ref{thm:without} for $\beta=0$. Also, $v\mathbb{I}\neq-R$, because 
$v\mathbb{I}=-R$ may correspond only to the Friedrichs extension (Proposition~\ref{prop:Infty2}) whose spectrum is 
only absolutely continuous. However, if one takes the coupling parameter of contact interaction of the form
$C=c\mathbb{I}$ with $c$ nonzero and real, then necessarily $R=r\mathbb{I}$ with $r$ real. It then follows that 
$v\sim -r$ for $\abs{c}$ arbitrarily large but finite. For a given $r$, there exists $\omega$ corresponding to $v\sim -r$,
so (\ref{eq:last}) shows that the perturbed Hamiltonian with arbitrarily large but finite impurity scattering 
strength $\abs{c}$ has an eigenvalue below the threshold. For example, when $\alpha=2$ and $\omega=0$, 
the parameter $-v=(\mrm{arcosh}(3)-2\sqrt{2})/(2\sqrt{2}+\pi)$ approximates $r$, and the eigenvalue 
$E\approx-1.43923$ (Fig.~\ref{fig:fig2}(a)). When $\omega=0$ and $\beta\sim0$, there also exists an eigenvalue 
$E_{\infty,1}\approx0.74018$ above the threshold (Theorem~\ref{thm:existence3}), but the same does not apply to
the case when $\beta=0$. This explains why we chose $\beta$ arbitrarily small but nonzero. However,
the above considerations remain valid for $\beta=0$ as well, with the exception that now the hypothesis of 
Theorem~\ref{thm:existence3} fails. For comparison reasons, in Fig.~\ref{fig:fig2}(b) we plot the eigenvalues 
corresponding to $s=\pm$ in (\ref{eq:sing}) as functions of $c$ when the magnetic field is nonzero. The present 
discussion is easy to extend to a general case when $\omega_{+}\neq \omega_{-}$ and $\gamma\neq0$. Moreover,
the considerations apply to the case when $\alpha<\sqrt{2\beta}$. One can show similar plots either solving
equation~(\ref{eq:sing}) numerically or applying analytic expressions in Theorem~\ref{thm:existence} for 
$\alpha$ sufficiently small. For the mean field analysis of eigenvalues at large $\abs{c}$, but in dimensions one 
and two, the reader may also refer to \cite{Kim15,Hu13,Liu13}, where the case when $\alpha<\sqrt{2\beta}$ is studied. 

We recall that in our model, when $\alpha\geq \sqrt{2\beta}$, $2$ is the maximum possible value of $\alpha$, as well as 
$1$ is the maximum possible value of $\beta$. This is due to the restriction $\Sigma\leq 1$ caused by the absolute 
convergence of the hypergeometric series which defines the Green function for the free Hamiltonian. However, an 
additional procedure of analytic continuation, similar to that when $\alpha<\sqrt{2\beta}$ is small, may probably lead 
to larger $\alpha$ as well. When applied to concrete schemes that propose the creation of spin-orbit coupling in cold 
atomic gases, the parameters $\alpha$ and $\beta$ are expressed in terms of the quantities which are measurable 
experimentally either directly or indirectly. For example, using $\hbar=1$ and $m=1/2$ ($m$ is the mass of the particle), 
the spin-orbit coupling for the Rashba Hamiltonian constructed in \cite{Anderson13} is characterized by 
$\alpha=k_{\mrm{eff}}/2$, where $k_{\mrm{eff}}$ ($\approx 1\mu$m$^{-1}$ in SI units) points to the strength of the 
magnetic field gradient. Without the additional procedure of continuation our model can therefore be applied 
straightforwardly up to $k_{\mrm{eff}}=4$. Likewise, in the experimental setup proposed in \cite{Cheuk12}, the 
spin-orbit coupling is given by $\alpha=\sqrt{E_{R}}/2$, where $E_{R}$ is the two-photon recoil energy due to 
Raman process. Hence, in the regime $\alpha\geq \sqrt{2\beta}$, one would expect $E_{R}\leq 16$.
In \cite{Liu13} one can find some realistic parameters suggested for observing bound states when 
using the scheme in \cite{Wang12} with $\alpha\approx2\sqrt{E_{R}}$ and $\beta \gtrsim 4.2E_{R}$.

While the point spectrum of the Rashba Hamiltonian with spin-dependent delta-like impurity scattering depends on
the admissible matrix $R$, the main conclusion, which is independent of $R$, is that: 
1) The Hamiltonian does not have eigenvalues above the threshold provided that the spin-orbit-coupling 
strength is small enough ($\alpha<\sqrt{2\beta}$) but nonzero.
2) The Hamiltonian does not have eigenvalues located in-between the threshold and the 
minimum of the upper branch of dispersion provided that the spin-orbit-coupling strength is large enough
($\alpha\geq \sqrt{2\beta}$).
3) The maximum possible eigenvalue that the Hamiltonian can reach is $E_{\nu,1}$, and $E_{\nu,1}\geq \beta$
depends only on $\alpha\geq \sqrt{2\beta}$ and $\beta>0$. 
\section*{Acknowledgments}
The author acknowledges the anonymous referees whose remarks have helped in improving the
presentation. The author also thanks the referee of the previous versions of the paper.
\providecommand{\bysame}{\leavevmode\hbox to3em{\hrulefill}\thinspace}
\providecommand{\MR}{\relax\ifhmode\unskip\space\fi MR }
\providecommand{\MRhref}[2]{%
  \href{http://www.ams.org/mathscinet-getitem?mr=#1}{#2}
}
\providecommand{\href}[2]{#2}

\end{document}